\documentclass[a4paper, reqno, 12pt]{amsart}
\usepackage{a4wide}
\usepackage[french, english]{babel}
\usepackage[normalem]{ ulem }
\usepackage{soul}

\usepackage[T1]{fontenc}
\usepackage[utf8]{inputenc}
\usepackage{amsmath,amsfonts,amssymb,amsopn,amscd,amsthm, mathrsfs}

\usepackage{gensymb}
\usepackage{comment}
\usepackage{dsfont}
\usepackage{bm}
\usepackage{graphicx}
\usepackage{xcolor}
\usepackage[colorlinks]{hyperref}
\usepackage{epigraph}
\usepackage{enumerate}
\usepackage{xargs}
\usepackage[prependcaption]{todonotes}

\usepackage{marginnote}
\usepackage{todonotes}

\title[ Quantum jumps via homogenization ]
      { Emergence of jumps in quantum trajectories via homogenization }

\author[Benoist]{\textsc{Tristan Benoist}}
\address{Institut de Math\'ematiques de Toulouse, UMR5219, Universit\'e de Toulouse, CNRS, UPS IMT, F-31062 Toulouse Cedex 9, France}
\email{{\tt tristan.benoist@math.univ-toulouse.fr}}

\author[Bernardin]{\textsc{C\'edric Bernardin}} 
\address{Universit\'e C\^ote d'Azur, CNRS, LJAD\\
Parc Valrose\\
06108 NICE Cedex 02, France\\
\& Interdisciplinary Scientific Center Poncelet (CNRS IRL 2615), 119002 Moscow, Russia}
\email{{\tt cbernard@unice.fr}}

\author[Chetrite]{\textsc{Rapha\"el Chetrite}}
\address{Universit\'e C\^ote d'Azur, CNRS, LJAD\\
Parc Valrose\\
06108 NICE Cedex 02, France\\}
\email{{\tt raphael.chetrite@unice.fr}}
      
\author[Chhaibi]{\textsc{Reda Chhaibi}}
\address{Institut de Math\'ematiques de Toulouse, UMR5219, Universit\'e de Toulouse, CNRS, UPS IMT, F-31062 Toulouse Cedex 9, France}
\email{{\tt reda.chhaibi@math.univ-toulouse.fr}}

\author[Najnudel]{\textsc{Joseph Najnudel}} 
\address{University of Bristol -- University Walk, Clifton, Bristol, United Kingdom}
\email{{\tt joseph.najnudel@bristol.ac.uk}}

\author[Pellegrini]{\textsc{Cl\'ement Pellegrini}}
\address{Institut de Math\'ematiques de Toulouse, UMR5219, Universit\'e de Toulouse, CNRS, UPS IMT, F-31062 Toulouse Cedex 9, France}
\email{{\tt clement.pellegrini@math.univ-toulouse.fr}}

\date{\today}

\allowdisplaybreaks[4]

\DeclareMathOperator{\tr}{tr}

\DeclareMathOperator*{\esssup}{ess\,sup}

\DeclareMathOperator{\Id}{Id}
\DeclareMathOperator{\Ker}{Ker}

\DeclareMathOperator{\End}{End}
\DeclareMathOperator{\Spec}{Spec}
\DeclareMathOperator{\Span}{Span}
\DeclareMathOperator{\diag}{diag}
\DeclareMathOperator{\id}{id}

\def\half{\frac{1}{2}}

\def\B{{\mathbb B}}
\def\N{{\mathbb N}}

\def\Q{{\mathbb Q}}
\def\R{{\mathbb R}}
\def\C{{\mathbb C}}

\def\D{\mathbb{D}}

\def\P{{\mathbb P}}
\def\E{{\mathbb E}}
\def\F{{\mathbb F}}

\def\Bc{{\mathcal B}}
\def\Cc{{\mathcal C}}

\def\Ec{{\mathcal E}}
\def\Fc{{\mathcal F}}

\def\Lc{{\mathcal L}}

\def\Oc{{\mathcal O}}
\def\Pc{{{\mathcal P}}}

\def\Rc{{\mathcal R}}
\def\Sc{{\mathcal S}}

\def\xb{{\mathbf x}}

\newtheorem{thm}{Theorem}[section]
\newtheorem{proposition}[thm]{Proposition}
\newtheorem{corollary}[thm]{Corollary}

\newtheorem{definition}[thm]{Definition}

\newtheorem{lemma}[thm]{Lemma}

\newtheorem{rmk}[thm]{Remark}
\newtheorem{assumption}[thm]{Assumption}

\numberwithin{equation}{section}
\numberwithin{figure}{section}

\renewcommand{\Re}{\operatorname{Re}}
\renewcommand{\Im}{\operatorname{Im}}

\renewcommand{\imath}{i}

\let\oldtocsection=\tocsection
 
\let\oldtocsubsection=\tocsubsection
 
\let\oldtocsubsubsection=\tocsubsubsection
 
\renewcommand{\tocsection}[2]{\hspace{0em}\oldtocsection{#1}{#2}}
\renewcommand{\tocsubsection}[2]{\hspace{2em}\oldtocsubsection{#1}{#2}}
\renewcommand{\tocsubsubsection}[2]{\hspace{4em}\oldtocsubsubsection{#1}{#2}}

\begin{document}

\begin{abstract}
In the strong noise regime, we study the homogenization of quantum trajectories i.e. stochastic processes appearing in the context of quantum measurement.

When the generator of the average semigroup can be separated into three distinct time scales, we start by describing a homogenized limiting semigroup. This result is of independent interest and is formulated outside of the scope of quantum trajectories.

Going back to the quantum context, we show that, in the Meyer-Zheng topology, the time-continuous quantum trajectories converge weakly to the discontinuous trajectories of a pure jump Markov process. Notably, this convergence cannot hold in the usual Skorokhod topology.
\end{abstract}

\medskip

\renewcommand{\subjclassname}{%
  \textup{2010} Mathematics Subject Classification}

\maketitle

\newpage
\hrule
\tableofcontents
\hrule
\newpage


\section{Introduction}

Let us start by considerations from quantum mechanics which motivate the stochastic differential equations (SDEs) studied in this paper, as well as their strong noise limits.

\subsection{Physical motivations}
\
\medskip

{\bf Semigroups associated to open quantum systems: } In quantum optics, the evolution of a $d$-level atom is often described using a Markov approximation. Then the system state, encoded into a $d\times d$ density matrix (\emph{i.e.} a positive semidefinite matrix of trace $1$), evolves by the action of a semigroup of completely positive\footnote{Completely positive maps $\Phi:M_d(\C)\to M_d(\C)$ are linear maps such that for any $n\in \mathbb N$, $\Phi\otimes \Id_{M_n(\C)}:M_d(\C)\otimes M_n(\C)\to M_d(\C)\otimes M_n(\C)$ is positive.} trace preserving linear maps whose generator $\Lc$ is called a Lindbladian. More precisely, the evolution of the system's density matrix is solution of the linear ordinary differential equation (ODE)
\begin{equation}\label{eq:Master1}
d\bar\rho_t=\Lc(\bar\rho_t)dt,\quad \bar\rho_0\in \{\rho\in M_d(\C): \rho\geq 0,\quad  \tr \rho=1\}.
\end{equation}
Such equations are known as \emph{quantum master equations}. In a typical quantum optics experiment, one may identify three different contributions to the evolution of the atom. A first contribution is the Hamiltonian dynamic that an experimenter would like to realize {\footnote{The Hamiltonian dynamics can also be intrinsic, independent of any additional drive due to the experimentalist.}}. A second one is the unavoidable environment perturbation that often leads the atom to a steady state. The third one is the effect of any instrument that the experimentalist may put in contact with the atom to track its state. For more details, we refer the reader to \cite{BRE02}.

In this article we are interested in situations where the dynamics generator, $\Lc\equiv\Lc_\gamma$, is associated to three well separated time scales. The separation is done through some parameter $\gamma>0$:
$$\Lc_\gamma = \Lc^{(0)}
             + \gamma   \Lc^{(1)}
             + \gamma^2 \Lc^{(2)} \ .$$
To motivate such a setting, let us consider experiments similar to the famous one realized by Haroche's group \cite{Haroche}. In such experiments the aim is to track the unitary dynamic of a $d$-energy level quantum system when it is well-isolated from its environment. The dynamic induced by the environment is modeled by $\Lc^{(0)}$, the unitary dynamic by $\Lc^{(1)}$ and the effect of the instrument by $\Lc^{(2)}$. Here the large $\gamma$ limit corresponds to a fast decoherence, at speed $\gamma^2$, induced by the instrument compared to the slower steady state relaxation induced by the environment, with speed $\gamma^0=1$. To counteract the Zeno effect, the relevant scale of the unitary dynamic is the intermediary speed $\gamma^1=\gamma$. 

This choice of scaling of the Lindbladian is not limited to such experimental situations.  For different examples of dynamics verifying our choice of scales, see \cite[Section 4.3]{percival1998}.

\medskip

{\bf Stochastic semigroups in the presence of measurements:}
Equation \eqref{eq:Master1} only describes the evolution of a quantum system without reading measurement outcomes coming from the instruments. Taking them into account leads to a stochastic process $\rho^\gamma = \left( \rho_t^\gamma \ ; \ t \geq 0 \right)$ called a quantum trajectory and which takes values in density matrices. This process is solution to an SDE called a \emph{stochastic quantum master equation}. The drift part of this SDE is given by $\Lc_\gamma (\rho_t^\gamma)$. The noise part results from conditioning upon the measurement outcomes. Such models are often used to describe experiments in quantum optics -- see \cite{wisemanmilburn,BRE02}. In the present article we limit ourselves to diffusive quantum trajectories. In that case the SDE takes the form, in the Itô convention,
\begin{equation}\label{eq:def_QTraj1}
d\rho_t^\gamma=\Lc_\gamma (\rho_t^\gamma)+\sigma_\gamma(\rho_t^\gamma)dW_t,
\end{equation}
where the volatility $\sigma_\gamma$ is a quadratic function of density matrices (see Section \ref{sec:def_QTraj} for the exact expression) and $W$ is a standard Brownian motion. The average evolution of $\rho^\gamma$ solution of \eqref{eq:def_QTraj1} is given by the solution $\overline{\rho}^\gamma$ of \eqref{eq:Master1}.

SDEs such as \eqref{eq:def_QTraj1} where first introduced as effective stochastic models for wave function collapse -- see \cite{gisin,pearle1984comment,diosi} and references therein. They were used as well as their Poisson noise version as a numerical tool to compute the average evolution $\overline{\rho}^\gamma$ in \cite{dalibard, GisinPercival1992}. Since then, different justifications were given for the fact that they model quantum systems which are subject to continuous indirect measurements. Historically, the first one is based on quantum stochastic calculus and quantum filtering \cite{belavkin89}. In that setting, the interaction of the measurement apparatus and the environment with the open system is unitary and described by a quantum SDE \cite{par92}. We refer the reader to \cite{bouten} for an accessible introduction to quantum filtering.

 A second, more phenomenological, approach \cite{barchielliholevo,BarchielliGregoratti09} starts with a linear SDE extending the deterministic linear equation \eqref{eq:Master1}. By normalizing the resulting process in the set of positive semidefinite operators, and after a Girsanov transform, one obtains the SDE \eqref{eq:SDE} for density matrices.
 
  Another approach is based on the continuous-time limit of fast quantum repeated  measurements. 
  Introducing proper scaling, discrete time quantum trajectories converge weakly, in the continuous time limit, towards processes solution of SDEs such as \eqref{eq:SDE} -- see \cite{pellegrini3, bbbcontinuous} and references therein.
  

\subsection{Strong noise limits in the literature}
The large $\gamma$ limit for quantum trajectories has recently attracted a lot of attention. This limit is interpreted as a strong measurement regime. Let us discuss the phenomenon with the guiding intuition of Fig.~\ref{fig:threeState}. Interestingly while quantum trajectories are continuous in time for any finite $\gamma>0$, in the strong noise limit $\gamma \rightarrow \infty$, one observes a limiting process concentrated on particular states and which is jumping between them. These states are called pointer states. See for example \cite{Haroche,flindt2009universal} for experiments and \cite{bauer2015computing,ballesteros2017} for theoretical works. Surprisingly, these jumps do not come alone. They are decorated by spikes (or shards) as made explicit in \cite{tilloy2015spikes}. In a $2$-level system, the existence of these spikes has been extensively studied (see \cite{tilloy2015spikes,bauer2016zooming,bauer2018stochastic} and recently in \cite{bernardin2018spiking}).

\begin{figure}[htp!]
\includegraphics[scale=0.32]{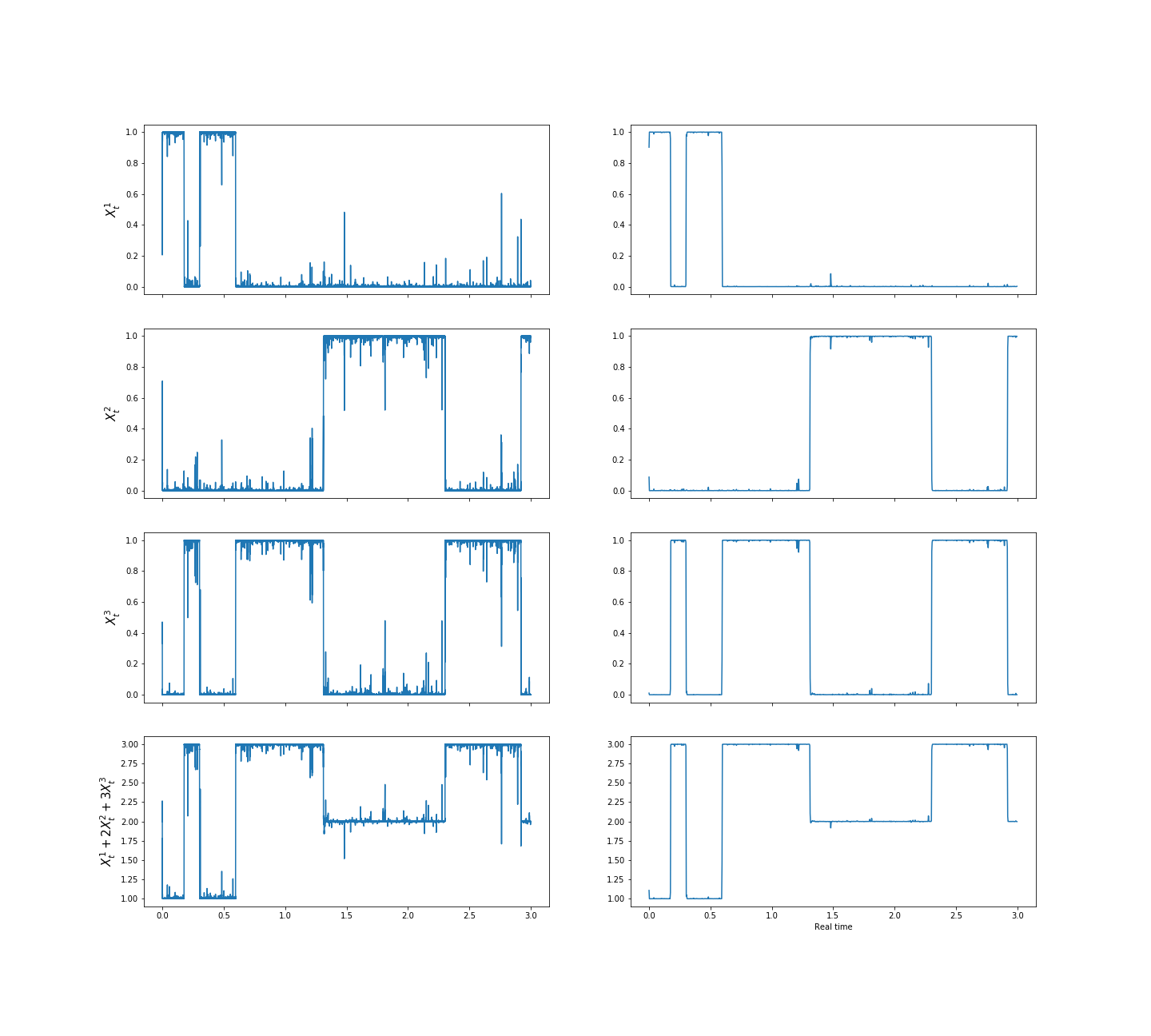}
\caption{Numerical simulation of the diagonal process $X^\gamma = \left( \rho^\gamma_{i,i} \ ; \ i=1,2,3 \right)$ on the left, and its smoothing on the the right. We have $\gamma = 10^4$ and there are $10^6$ time steps. Smoothing is via averaging over $1000$ steps. The first three rows display each coordinate $X^\gamma_{i}$ for $i=1,2,3$. The last row displays $X^\gamma_{1} + 2 X^\gamma_{2} + 3 X^\gamma_{3}$ which concentrates most of time on $\{1,2,3\}$ -- as $X^\gamma$ concentrates the corresponding point of the pointer basis $\{e_1,e_2,e_3\}$. The code is available at the online repository
\newline
\url{https://github.com/redachhaibi/quantumCollapse} \ 
\newline
\newline
In the notation of the upcoming SDE \eqref{eq:SDE}, with $d=3$, we take no Hamiltonian i.e. $H^{(0)}=H^{(2)}=0$, no intermediary scale i.e. $\Lc^{(1)}=0$, and for $\Lc^{(2)}$, a single Kraus operator i.e. $\ell_2 = 1$ with efficiency $\eta_{\alpha=2}(1)=1$, given explicitly  by $L_1^{(2)}=\diag(1,2,3)$. Regarding the lowest order term, $\Lc^{(0)}$ is defined in terms of the $\ell_0=9$ Kraus operators $L_k^{(0)}$ given by the $9$ elementary matrices $E_{i,j}$, $1\le i,j\le 9$, and we take all the efficiencies $\eta_{\alpha=0}(k)$ equal to $0$. In this weak-coupling set-up \cite{davies1976markovian}, the diagonal $X^\gamma$ of $\rho^\gamma$ evolves autonomously according to the SDE $dX^\gamma_t = R X_t^\gamma dt + 2 \gamma \left[ L_1^{(2)} X_t^\gamma - \left\langle L_1^{(2)} X_t^\gamma , \mathds{1} \right\rangle X^\gamma_t  \right] dW_t$, where $(W_t)_{t\ge 0}$ is a standard Brownian motion, $R$ is the $3$-square matrix such that $R_{i,j}= 1 -3\delta_{i,j}$ and $\mathds{1}$ is the vector with all coordinates equal to one.}
\label{fig:threeState}
\end{figure}

A crucial issue in the question of convergence of the quantum trajectories to a jump process is the choice of an appropriate topology. The usual Skorokhod topology on càdlàg functions appears to be useless. Indeed, without even taking into account the spikes, there is the following basic obstruction. As underlined by \cite[Theorem 13.4]{billingsley2013convergence}, weak convergence of processes with continuous paths in the Skorokhod topology yields a limiting process with continuous paths. As a consequence, there is no hope to obtain a limit in this excessively strong topology.

The models in \cite{bernardin2018spiking} are focused on $2$-level systems. One crucial feature is that these models can be recast into $1$-dimensional diffusions $X^\gamma = \left( X^\gamma_t \ ; \ t \geq 0 \right)$. In particular the main result of \cite{bernardin2018spiking}, stated in Theorem 2.1, proves a two-fold convergence. The first half of the theorem states that the process $X^\gamma$ converges in an averaged sense to a two-state Markov jump process $X$. The exact statement is essentially equivalent to:
\begin{align}
\label{eq:smoothed_cv}
\lim_{n \to \infty}\int_0^\infty f(t, X_t^\gamma)e^{-t}dt & = \int_0^\infty f(t, X_t) e^{-t}dt \quad \textrm{a.s.},
\end{align}
for any continuous bounded function $f$. The second half of the theorem states that the graph of $X^\gamma$ converges in the Hausdorff topology to an explicit random set, which captures the spikes in the strong $\gamma$ regime. Both convergences are almost sure thanks to an ad-hoc but convenient coupling. Moreover, their proofs rely heavily on the use of local times, scale functions and the Dambis--Dubins--Schwarz Theorem, which are one-dimensional techniques for diffusions. Unfortunately, these tools are not adapted to the higher dimensional problem, which is still a challenge.

\subsection{Our contribution}
In the present paper, we generalize to arbitrary finite dimensions the first half of the aforementioned \cite[Theorem 2.1]{bernardin2018spiking}, that is to say the convergence towards the jump process between pointer states. In order to have an intrinsic proof, it is desirable to invoke the classical machinery of weak convergence of stochastic processes and to avoid using any coupling. Also, since we focus only on the jumps between the states, the spikes need to be discarded. As shown in \cite{bauer2016zooming,bernardin2018spiking}, only countably many spikes appear in the limit, each being infinitely thin. As such, spikes are of zero Lebesgue measure and disappear upon averaging in Eq.~\eqref{eq:smoothed_cv}. Therefore, an ideal candidate for this task is the topology of convergence in (Lebesgue) measure. 


The study of the convergence in law of stochastic processes in this topology was pioneered by Meyer and Zheng in \cite{MZ} -- see \cite{Kurtz} for further developments and \cite{rebolledo1987topologie} for an application to weak noise limits. As we shall see, a path $X^\gamma$ converges to $X$ in the Meyer-Zheng topology if and only if Eq.~\eqref{eq:smoothed_cv} holds. 
This topology is also called pseudo-paths topology and is much weaker than the usual Skorokhod topology. 

Our main result is stated in Theorem \ref{thm:main}. It shows that in the Meyer-Zheng topology, in the limit of large $\gamma$, the quantum trajectory we study converges in law to a Markov process on the pointer states with explicit rates. Not only this provides an extension but also a mathematically complete and rigorous proof of the pioneering works of \cite{bauer2015computing}.

We also establish a general homogenization result for semigroups on finite-dimensional Hilbert spaces that is instrumental to the proof of Theorem \ref{thm:main}. In the usual homogenization references such as \cite{P78,CD99,pavliotis2008multiscale, BLP}, there is a trivial distinction  between  a slow and a fast variable and it is then assumed that by fixing the slow variable the fast process is ergodic. The novelty of our homogenization result is that it holds for abstract semigroups and moreover the state space is not a priori the direct product of slow and fast variables. In particular, we show that it applies to semigroups generated by Lindbladians $\Lc_\gamma$. 

\subsection{Structure of the paper}

The article is structured as follows. 

In Section \ref{sec:conv_traj}, we start by defining the mathematical objects we study and state our two working Assumptions \ref{ass:(Nd)} and \ref{ass:(Id)}. Then, in Subsection \ref{subsec:main} we introduce the Meyer-Zheng topology in Definition \ref{def:MZtopology} and state our main result in Theorem \ref{thm:main}. We conclude the section with some remarks on the main result. 

Section \ref{section:homogenization} is devoted to the abstract homogenization result and its application to Lindbladians.

We finally prove our main result in Section \ref{section:proof}.

\section{Main result}
\label{sec:conv_traj}

\subsection{Notation}
\label{subsec:not}
We denote by $\langle\cdot, \cdot \rangle$ the standard scalar product on $\C^d$ and $\Vert \cdot\Vert$ the corresponding norm, $M_d(\C)$ the set of $d\times d$ complex matrices, $X^*$ the conjugate transpose of $X\in M_d(\C)$. The Hilbert-Schmidt inner product $(X,Y)\in M_d(\C)^2 \mapsto \langle X, Y\rangle :=\tr (X^* Y)$  transforms $M_d(\C)$ into a Hilbert space and the associated norm is also denoted by $\Vert \cdot \Vert$. The set $\Sc=\{\rho\in M_d(\C) \ | \ \rho\geq 0, \ \tr\rho=1\}$ is a compact convex set whose elements are called density matrices. For any two matrices $X, Y$ of $M_d(\C)$, $[X,Y] := XY - YX$ is the commutator while $\{X, Y\} := XY + YX$ is the anti-commutator. An endomorphism $\Lc$ on $M_d (\C)$ is sometimes called a super-operator while a matrix $X\in M_{d} (\C)$ is called an operator. The algebra of super-operators $\left({\rm{End}} (M_d(\C)), +, \circ\right)$ is equipped with the operator norm (with respect to the Hilbert-Schmidt norm on $M_d (\C)$) and denoted also by $\Vert \cdot \Vert$. We usually reserve the notation $\circ$ for super-operators to emphasize the distinction with operators. The adjoint w.r.t the Hilbert-Schmit scalar product of a linear operator $\Lc: M_d (\C) \to M_d (\C)$ is denoted by $\Lc^*$. For $x\in \C^\ell$ and $A:=(A_k)_{k=1}^\ell \in M_d(\C)^\ell$, we denote $A\cdot x=\sum_{k=1}^\ell A_k x_k$ and the action of $\Lc \in {\rm{End}} (M_d (\C))$ on such $A$ is understood component-wise, i.e. $\Lc (A)= (\Lc (A_k))_{k=1}^\ell$. If $(a_t)_{t \in \R_+}$ is a continuous-time stochastic process defined on a probability space $(\Omega, {\mathcal F}, {\mathbb P})$ we denote $\| a \|_\infty := \esssup_{\omega \in \Omega} \sup_{t \in \R_+} |a_t (\omega)|$.

\subsection{Definitions}

\subsubsection{Lindbladians}
By definition a Lindbladian $\Lc: M_d(\C) \rightarrow M_d(\C)$ is the generator of a continuous semigroup of completely positive trace-preserving maps, which describes the Markovian evolution of a quantum open system. 
Following \cite{gorini,lindblad}, a Lindblad super-operator admits a GKSL\footnote{The acronym GKSL stands for Gorini-Kossakowski-Sudarshan-Lindblad.} decomposition i.e. for all $X \in M_d (\C)$:
\begin{align}
\label{def:gksl}
\Lc(X) = & \ -\imath [H,X]+\sum_{k=1}^\ell  \left(L_k XL_k^*-\tfrac12\{L_k^*L_k,X\} \right) \ ,
\end{align}
where $(H,(L_k)_{k=1}^\ell)$ are matrices of $M_{d} (\C)$ such that $H^*=H$. We call the first matrix, $H$, the Hamiltonian and the operators $(L_k)_{k=1}^\ell$, Kraus operators{\footnote{By analogy with discrete quantum channels.}}.

\subsubsection{Diffusive quantum trajectories with three time scales}\label{sec:def_QTraj}
In this paper, for any $\gamma>0$, we consider diffusive quantum trajectories given by the It\^o SDE
\begin{align}
\label{eq:SDE}
d\rho^\gamma_t = & \ \Lc_\gamma(\rho^\gamma_{t})dt + \sum_{\alpha=0,1,2} \gamma^{\frac{\alpha}{2}} \sigma^{(\alpha)} (\rho^\gamma_t)\cdot dW^\alpha_t\\
               = & \ \Lc_\gamma(\rho^\gamma_{t})dt
                   + \sigma^{(0)} (\rho^\gamma_t)\cdot dW^0_t
                   + \gamma^{\frac{1}{2}} \sigma^{(1)} (\rho^\gamma_t)\cdot dW^1_t
                   + \gamma \sigma^{(2)}  (\rho^\gamma_t)\cdot dW^2_t
                   \nonumber
\end{align}
with initial condition $\rho^\gamma_0  = \varrho \in {\mathcal S}$. Throughout the paper, the Itô convention for SDEs is in place. The drift $\Lc_\gamma(\rho^\gamma_{t})$ of this equation is given by the Lindblad super-operator $\Lc_\gamma$ having the form  
\begin{equation}
\label{eq:perturbed_lind}
\Lc_\gamma := \Lc^{(0)} + \gamma \Lc^{(1)} + \gamma^2 \Lc^{(2)} \ .
\end{equation}

We  denote by $(H^{(0)},(L^{(0)}_k)_{k=1}^{\ell_0})$, $(H^{(1)},(L^{(1)}_k)_{k=1}^{\ell_1})$ and $(H^{(2)},(L^{(2)}_k)_{k=1}^{\ell_2})$ the GKSL decompositions of the Lindbladians $\Lc^{(0)}$, $\Lc^{(1)}$ and $\Lc^{(2)}$ respectively.

In the noise part $\sum_{\alpha=0,1,2} \gamma^{\frac{\alpha}{2}} \sigma^{(\alpha)} (\rho^\gamma_t) \cdot dW^\alpha_t$, the processes $\left( W^\alpha, \alpha=0,1,2 \right)$ are independent $\ell_\alpha$-dimensional (standard) Wiener processes and the maps $\sigma^{(\alpha)}: \rho\in\Sc \mapsto\left(\sigma^{(\alpha)}_k(\rho) \right)_{k=1}^{\ell_{\alpha}} \in M_d(\C)^{\ell_\alpha}$ are the three quadratic maps defined component-wise by
\begin{align}
\label{eq:def_sigma}
\sigma_k^{(\alpha)}  (\rho) = & \ {\sqrt{\eta_\alpha (k)}}\, \left(L^{(\alpha)}_k \rho+\rho L^{(\alpha)\, *}_k-\tr[(L^{(\alpha)\, *}_k+L^{(\alpha)}_k)\rho]\, \rho \right)  \ , 
\end{align}
for $\alpha=0,1,2$ and $ k=1, \ldots, \ell_\alpha$. Here $(\eta_\alpha(k))_{k=1}^{\ell_\alpha}\in [0,1]^{\ell_\alpha}$ are given numbers. We refer the reader to the previous notation section (Section \ref{subsec:not}) for the meaning of $\sigma^{(\alpha)} (\rho^\gamma_t) \cdot dW^\alpha_t$.

The proof of existence and uniqueness of the strong solution to Eq.~\eqref{eq:SDE} can be found in \cite{barchielliholevo, pellegrini1, BarchielliGregoratti09, pellegrini3}. In these references, it is also proven that $\rho^\gamma\in C \left( {\mathbb R}_+ \, ; \, \Sc\right)$ almost surely.

Since the SDE \eqref{eq:SDE} has a linear drift, it follows that the average evolution of $\rho^\gamma$ is expressed in terms of the semigroup generated by $\Lc_\gamma$:
 \begin{align}
\label{eq:expectation_rho}
\forall t\ge 0, \quad \E(\rho_t^\gamma) = & \ e^{t\Lc_\gamma}\varrho \ .
\end{align}
The asymptotic analysis of this average semigroup will in fact play a crucial role in the proof of the main result, see Proposition \ref{proposition:Lindblad_perturbation}.

In terms of interpretation of indirect measurement, the Wiener process $W^\alpha$ results from the output signal of measurements. The numbers $\eta_\alpha (k)$ are introduced in order to encapsulate in a single form the measurement and thermalization aspects. More precisely $\eta_\alpha (k)=1$ corresponds to perfectly read measurements, $\eta_\alpha (k) \in (0,1)$ to imperfectly read measurements and $\eta_\alpha (k)=0$ to unread measurements or to model contributions from a thermal bath.

\subsubsection{Assumptions}
Let us now state and discuss our working assumptions for the main result.
\begin{assumption}[Quantum Non-Demolition (QND) assumption]
\label{ass:(Nd)}
The operators $H^{(2)}$, $(L^{(2)}_k)_{k=1}^{\ell_2}$ and $(L^{(1)}_k)_{k=1}^{\ell_1}$ are all diagonalizable in a common orthonormal basis $(e_i)_{i=1}^d$ of $\C^d$, called the pointer basis. 
\end{assumption}

Observe that no assumption is made on the Hamiltonian $H^{(1)}$ nor on the Kraus operators and Hamiltonian involved in $\Lc^{(0)}$. Also, Assumption \ref{ass:(Nd)} is equivalent to requiring that the $\ast$-algebra generated by the Kraus operators $(L^{(2)}_k)_{k=1}^{\ell_2}$, the Hamiltonian $H^{(2)}$ as well as the Kraus operators $(L^{(1)}_k)_{k=1}^{\ell_1}$ is commutative.\footnote{In other words a $C^*$ commutative subalgebra of $M_d (\C)$.}

From a physical perspective, the QND assumption is standard. It is at the cornerstone of the experiment \cite{Haroche} where QND measurements are used to count the number of photons in a cavity without destroying them. It is shown that it reproduces the wave function collapse in long time -- see \cite{bauer2011convergence,bbbqnd,bpqnd} and references therein. This condition is tailored to preserve the pointer states during the quantum measurement process. More precisely, under the QND Assumption \ref{ass:(Nd)}, in the case $\Lc^{(0)}=\Lc^{(1)}=0$, if the initial state is a pointer state, i.e. $\varrho \in \{ E_{i,i} := e_i e_i^* \}_{i=1}^d$, then it is not affected by the indirect measurement in the sense that the state remains unchanged by the stochastic evolution~\eqref{eq:SDE}. Note that this behavior is very specific to such models since measurement usually induces a feedback on the quantum system. 

A simple computation detailed in Lemma \ref{lemma:lindblad_structure} below shows then that under the QND  Assumption \ref{ass:(Nd)}  the super-operator $\Lc^{(2)}: M_d(\C) \rightarrow M_d(\C)$ is diagonalizable with eigenvectors $\left( E_{i,j}:=e_i e_j^* \right)_{ i,j =1}^d$ and associated eigenvalues:
\begin{align}
\label{eq:L_eigenequation2}
 \tau_{i,j}  = \ & - \half \sum_{k=1}^\ell \left| (L^{(2)}_k)_{i,i} - (L^{(2)}_k)_{j,j} \right|^2 \\
                 & \quad \quad - \imath \left( H^{(2)}_{i,i} - H^{(2)}_{j,j}
                   + \sum_{k=1}^\ell \Im \left( \overline{(L^{(2)}_k)_{i,i}} (L^{(2)}_k)_{j,j} \right) \right)
                   \ \nonumber.
\end{align}

\medskip

Here and in the following, if $X \in M_d (\C)$, the notation $X_{i,j}$ always refer to the coordinates of $X$ in the pointer basis $(e_i)_{i=1}^d$. Observe also that the family $(E_{i,j})_{i,j=1}^d$ forms an orthonormal basis of $M_d (\C)$.

\begin{rmk}
In the literature two notions of quantum non-demolition exist. The notion we use refers to a measurement process where particular system-states (pointer states) are not affected (non-demolished) by the measurement process. This notion appeared in physics in the early eighties \cite{Braginsky547}. For a recent mathematical approach the reader can consult \cite{bbbqnd, bpqnd}. The other notion of non-demolition was introduced by Belavkin \cite{belavkin1992quantum, belavkin1994nondemolition} in the context of quantum filtering. There, non-demolition refers to commutative sets of operators which give observables that we condition upon. The resulting conditional probability is required for a properly defined quantum measurement process \cite{bouten, BoutenvHandelMatthew2009}.
\end{rmk}

\begin{assumption}[Identifiability condition]
\label{ass:(Id)}
For any $i,j\in\{1,\dotsc, d\}$ such that $i\neq j$, there exists $k\in \{1,\dotsc,\ell_2\}$ such that $\eta_2(k)>0$ and 
$$\Re ( L_k^{(2)} )_{i,i} \neq \Re ( L_k^{(2)} )_{j,j} \ .$$
\end{assumption}

In fact, from Eq.~\eqref{eq:L_eigenequation2}, this assumption together with the QND assumption imply the non-existence of purely imaginary eigenvalues $\tau_{i,j}$ for the super-operator $\Lc^{(2)}$. We shall see that this will play an important role.

Our motivation to qualify this assumption as identifiability originates again from the theory of non-demolition measurements. Indeed, following \cite{bbbqnd,bpqnd}, if the QND Assumption \ref{ass:(Nd)} and the identifiability condition of Assumption \ref{ass:(Id)} hold, for any $\gamma>0$, the quantum trajectory obtained when setting $\Lc^{(0)}=\Lc^{(1)}=0$ converges almost surely, as $t$ grows, to a random pointer state, reproducing a non-degenerate projective measurement along the pointer basis. If the identifiability Assumption \ref{ass:(Id)} does not hold, the limiting random state may exist but will correspond to a degenerate measurement.

\subsection{Statement}\label{subsec:main}

Before stating the main result of this article, which addresses the weak convergence as $\gamma$ goes to infinity of $(\rho^\gamma)_{\gamma>0}$ to a pure jump Markov process, we need to introduce the topological setting in which this convergence will hold. 

\begin{definition}[Meyer-Zheng topology]
\label{def:MZtopology}
Consider a Euclidean space $(E, \Vert\cdot\Vert)$ and denote by ${\mathbb L}^0:={\mathbb L}^0 (\R_+ ;  E)$ the space of $E$-valued Borel functions  on $\R_+$\footnote{To be more precise ${\mathbb L}^0$ is a quotient space where two functions are considered as equal if they coincide almost everywhere with respect to the Lebesgue measure.}. Given a sequence $(w_n)_{n\ge 0}$ of elements of ${\mathbb L}^0$, the following assertions are equivalent and define the convergence in Meyer-Zheng topology of  $(w_n)_n$ to $w\in {\mathbb L}^0$:
\begin{itemize}
\item For all bounded continuous functions $f: \R_+ \times E \rightarrow \R$, 
$$
\lim_{n \rightarrow \infty}
\int_0^\infty f\left(t, w_n(t) \right) e^{-t} dt \ 
=
\int_0^\infty f\left(t, w(t) \right) e^{-t} dt \ .
$$
\item For $\lambda(dt) = \ e^{-t} dt$, we have that for all $\varepsilon > 0$, 
$$ \lim_{n \rightarrow \infty}
   \lambda\left( \left\{ s \in \R_+ \ | \ \left\Vert w_n(s) - w(s) \right\Vert \geq \varepsilon \right\} \right)
   = 0 \ ,
$$ 
\item $ \lim_{n \rightarrow \infty} {\rm d} \left( w_n, w \right) = 0$
where $\rm d$ is defined by
\begin{equation}
\label{eq:distanced}
{\rm d} (w,w') := \int_0^\infty \Big\{1 \wedge \left\| w(t) - w'(t) \right\| \Big\} \ e^{-t} \ dt \ .
\end{equation}
\end{itemize}
The distance ${\rm d}$ metrizes the Meyer-Zheng topology on ${\mathbb L}^0$ and $({\mathbb L}^0, {\rm d})$ is a Polish space.

\end{definition}

\begin{proof}[Pointers to the proof]
The equivalence between the two first statements is \cite[Lemma 1]{MZ}. The equivalence with the third statement is a standard exercise.
\end{proof}

Notice that the first statement in Definition \ref{def:MZtopology} is exactly Eq.~\eqref{eq:smoothed_cv}, demonstrating that we have the correct setting for:

\begin{thm}[Main theorem]
\label{thm:main}
For any $\gamma>0$ let $\rho^\gamma$ be the continuous processes on $\Sc$ solution of Eq.~\eqref{eq:SDE} starting from $\varrho$. Under the QND Assumption \ref{ass:(Nd)} and the identifiability condition in Assumption  \ref{ass:(Id)}, we have{\footnote{We recall that if $\chi$ is a topological space and $(X_n)_n$ is a sequence of $\chi$-valued random variables, we say it converges weakly (or in law) to the $\chi$-valued random variable $X$ if and only if for any bounded continuous function $f:\chi \to \mathbb R$, $\lim_{n \to \infty} {\mathbb E} \big[ f(X_n)\big] ={\mathbb E} \big[ f(X)\big]$. }}:
$$ 
   \lim_{\gamma\to\infty} \rho^\gamma = \xb \xb^* ,\quad \textrm{weakly in }
   \left( {\mathbb L}^0 (\R_+;M_d(\C)), \ {\rm d} \right) 
   \ ,
$$
where $\xb := \left( \xb_t \ ; \ t \geq 0 \right)$ is a pure jump continuous-time Markov process on the pointer basis $(e_i)_{i=1}^d$ with initial distribution $\mu_{\varrho}$ defined by 
$$ \mu_{\varrho}: e_i \mapsto \langle e_i,\varrho e_i\rangle \ .$$

Furthermore, the generator $T$ of the Markov process $\xb$ is explicit. The transition rate{\footnote{Observe that the transition rate is independent of the $L^{(1)}_k$'s and of $H^{(0)}$. It depends of course of $H^{(1)}, L^{(0)}_k, L_k^{(2)}$ but also of $H^{(2)}$ through the eigenvalues $\tau_{i,j}$ given in Eq.~\eqref{eq:L_eigenequation2}.}} from $e_i$ to $e_j$, $i\ne j$, is given by  
\begin{equation}
\label{eq:T-rates}
T_{i,j}
   = \sum_{k=1}^{\ell_0} |( L_k^{(0)})_{j,i}|^2
   + \frac{\left| \left( H^{(1)} \right)_{i,j} \right|^2}{|\tau_{i,j}|^2} \sum_{k=1}^{\ell_2} \left| (L_k^{(2)})_{i,i} - (L_k^{(2)})_{j,j} \right|^2 \ .
 \end{equation}    
Here $\tau_{i,j}$ is the eigenvalue of $\Lc^{(2)}$ corresponding to the eigenvector $E_{i,j}=e_i e_j^*$ given in Eq.~\eqref{eq:L_eigenequation2}.
\end{thm}

\begin{proof}[Strategy of proof and structure of the paper] The approach is structured as follows.  

In Section \ref{section:homogenization} we state the general homogenization Theorem \ref{thm:homogenization} for semigroups. The proof follows the philosophy pioneered by Nakajima-Zwanzig. Then we apply it to the case of Lindblad super-operators in Proposition \ref{proposition:Lindblad_perturbation}. Let us mention \cite[Theorem 2.2]{ballesteros2017} as an inspiration for the proof and that our result is consistent with \cite{macieszczak2016towards,albert2016geometry}. There, we show that in the large $\gamma$ limit, the dynamic of the semigroup $e^{t\Lc_\gamma}$ reduces to a dynamic generated by an operator $\Lc_\infty$ whose expression is explicitly given in terms of  $\Lc^{(0)}$, $\Lc^{(1)}$ and $\Lc^{(2)}$. Thanks to Eq.~\eqref{eq:expectation_rho}, this leads to the convergence of the mean $\E\left( \rho^\gamma_t \right)$. Although this may seem to be very partial information, it is sufficient to identify the generator $T$.

In Section \ref{section:proof} we give the proof of our Main Theorem \ref{thm:main}. The proof follows the usual approach for the weak convergence of stochastic processes: we use a tightness criterion in the Meyer-Zheng topology and then identify the limit via its finite-dimensional distributions. Interestingly, the convergence of the mean is bootstrapped to the convergence of finite-dimensional distributions thanks to the Markov property and the collapsing on pointer states $( E_{i,i} )_{i=1}^d$. A more detailed sketch of the proof is given at the beginning of Section~\ref{section:proof}.
\end{proof}

\subsection{Further remarks}\label{sec:further remarks}
\ 
\medskip

{\bf Convention on $\Lc$ and $T$:} The Markov generator $T$ defined in Eq.~\eqref{eq:T-rates} follows the usual probabilistic convention in the sense that $T \mathds{1}=0$. On the contrary, to simplify notations, for the various Lindblad generators $\Lc$, we use the convention that they generate trace-preserving maps, thus their duals with respect to the Hilbert--Schmidt inner product verify $\Lc^*(\id)=0$, which is equivalent to $\tr\circ\Lc=0$.

\medskip

{\bf On finite-dimensional distributions: }
By definition of the weak convergence in ${\mathbb L}^0 (\R_+;M_d(\C))$, Theorem \ref{thm:main} is equivalent to saying that for all continuous bounded functions $f:\R_+ \times \Sc\to \R$,
$$\lim_{\gamma\to\infty}\int_0^\infty f(t, \rho^\gamma_t)\ e^{-t} dt=\int_0^\infty f(t, \xb_t \xb_t^*) \ e^{-t} dt\quad\text{weakly}.$$
Moreover, we have convergence of the finite-dimensional distributions of $\rho^\gamma$ in (Lebesgue) measure. More precisely, following \cite[Theorem 6]{MZ}, $(t_1,\dotsc,t_r)\mapsto\E(f(\rho_{t_1}^\gamma,\dotsc,\rho_{t_r}^\gamma))$ converges to $(t_1,\dotsc,t_r)\mapsto\E(f(\xb_{t_1}\xb_{t_1}^*,\dotsc, \xb_{t_r}\xb_{t_r}^* ))$ in ${\mathbb L}^1 (\R_+^r , \lambda^{\otimes r})$, for any continuous bounded function $f$ on $\Sc^r$. This convergence of finite-dimensional distributions is a rigorous formulation of the convergence stated in \cite{bauer2015computing}. This is detailed in Subsection~\ref{subsection:fd_distributions}.

\medskip

{\bf On the Meyer-Zheng topology:}
As mentioned in the introduction of \cite{MZ}, the space $\mathbb D$ of c\`adl\`ag functions is not Polish for this topology, since it is not even closed. Hence, one cannot invoke Prokohrov's Theorem in this set. However, the larger set of measurable functions is complete with respect to the distance $\rm d$. 

This brings forth another issue which is that weak limits for the Meyer-Zheng topology are not necessarily supported in $\mathbb D$ in general. In fact the central result in \cite{MZ} is a tightness criterion under which the weak limit of a given family of c\`adl\`ag processes is garanteed to be also c\`adl\`ag, which is exactly tailored for our need.

\medskip

{\bf Generalizations: }
Let us mention two possible extensions of the setting of this paper. In principle, our results and methods of proof carry to these cases mutatis mutandis. However, such extensions are not included as this would considerably decrease the readability of the paper.

In Eq.~\eqref{eq:perturbed_lind}, one could consider a further dependence in $\gamma$ by replacing $\Lc^{(0)}$ by $\Lc^{(0)}_\gamma$ such that a limit holds as $\gamma$ goes to infinity.

Also, throughout the paper, we limit ourselves to diffusive noises. But the setting can be extended to include Poisson noises in the SDE \eqref{eq:SDE}. In passing, let us mention an interesting result in this direction using a completely different method. In the particular case $\Lc^{(0)}=0$ and for Poisson noises only, instead of Wiener ones, an analogous result to the Main Theorem \ref{thm:main} is possible, building on \cite[Theorem 2.3 item (b)]{ballesteros2017}. Indeed, in that article, it is obtained that there exists $C>0$ such that for any $t>0$, $\E(\|\rho_t^\gamma-Y_t^\gamma\|)\leq C\gamma^{-\frac12}\sqrt{|\log\gamma|}$ with $(Y^\gamma)_{\gamma >0}$ converging in law to $\xb \xb^*$, in Skorokhod's topology, as $\gamma \to \infty$. Since $\Sc$ is compact, integrating both side of the inequality with respect to the probability measure $\lambda(dt) = e^{-t} dt$ on $\R_+$, the result follows from ${\mathbb L}^1\left( \R_+, \lambda \right)$ convergence.

\medskip

{\bf The noise vanishes on pointer states: } The following intuition dictates that the noise vanishing on pointer states is crucial in order to have emergence of jump processes from strong noise limits as in Theorem \ref{thm:main}. The idea is that, as $\gamma$ grows larger, the process $\rho^\gamma$ will spend more time in a thin layer around the points where the noise vanishes. Because of the QND Assumption \ref{ass:(Nd)} and the structure of the maps $\sigma_\alpha$ given in Eq.~\eqref{eq:def_sigma}, the noise in the SDE \eqref{eq:SDE} vanishes exactly on the pointer states $( E_{i,i} )_{i=1}^d$, hence the intuition of a limiting process taking values in $( E_{i,i} )_{i=1}^d$. 

%

\medskip

{\bf Reversibility properties of the limit generator $T$:} In the context of dynamics where there is a clear distinction between slow variables and fast variables, consider the reduced dynamic in slow variables, obtained by the elimination of fast variables by homogenization. A common belief in statistical physics is that such a reduced dynamic is in general ``more irreversible'' than the initial dynamic \cite{MACKEY,LEBOWITZ99,GORBAN-KARLIN,Lavis,BALIAN2005} -- and regardless of the reversibility of this initial dynamic. The seminal example is provided by the (irreversible) Boltzmann equation which is derived by a kinetic limit from a (reversible) microscopic dynamic ruled by Newton's equations of motion \cite{Cer88}. 

In our context, from both mathematical and physical perspectives, this leads naturally to ask the question of the links between the reversibility properties of the SDE \eqref{eq:SDE} and the reversibility properties of our effective Markov process  $\xb := \left( \xb_t \ ; \ t \geq 0 \right)$. This SDE is generically non-reversible but it may happen, in various situations, that the effective Markov process however is, highlighting a possible moderation to the aforementioned popular belief, at least in this particular quantum context. Indeed, it is for example easy to check  that if $H^{(1)}=0$ and there exists a probability $p:=(p_i)_{i=1}^d$ such that for any $1\le k \le \ell_0$, $p_i \vert \big( L_k^{(0)} \big)_{j,i}\vert^2= p_j \vert \big( L_k^{(0)} \big)_{i,j}\vert^2$, then $T$ is reversible with respect to the probability $p$, while the SDE is not. Therefore it would be interesting to understand what are the conditions to impose on the Kraus and Hamiltonian operators, and more importantly their physical meaning, in order to obtain a reversible $T$. In the previously mentioned example, the condition is reminiscent of the one resulting from a weak coupling limit \cite{ALICKI, Davies, AL07} of a quantum system interacting with a heat bath at thermal equilibrium, showing this condition has probably some deeper physical interpretation.

\section{On the homogenization of semigroups}
\label{section:homogenization}

In this section, we establish a general homogenization result which does not rely specifically on the Lindbladian structure of the operators involved.  Therefore in the next Subsection \ref{subsection:homogenization_thm}, the symbol $\Lc$ will denote generic linear operators on a finite-dimensional vector space $V$ which are not necessarily Lindbladians. 

\subsection{A general statement of independent interest}
\label{subsection:homogenization_thm}

To simplify notations the symbol $\circ$ for the composition of (super)-operators is not used in this subsection. 

\begin{thm}[Abstract homogenization of semigroups]
\label{thm:homogenization}
Let $V$ be a finite-dimensional Hilbert space. Let $\gamma>0$ and $\Lc^{(0)}, \Lc^{(1)}, \Lc^{(2)}: V \rightarrow V$ be linear operators and define:
$$ \Lc_\gamma := \Lc^{(0)} + \gamma \Lc^{(1)} + \gamma^2 \Lc^{(2)}.$$
Consider the following assumptions:
\begin{enumerate}[i)]

\item {\rm{(Spectral property)}} The dominant operator $\Lc^{(2)}$ has no purely imaginary eigenvalues.



\item {\rm{(Contractivity property)}} We have at least one of the two conditions which holds for some norm $\Vert \cdot \Vert_\infty$ on $V$\footnote{Here also we denote with the same notation the corresponding operator norm on ${\rm{End}} (V)$.}:
\begin{equation}
\label{eq:provisoire}
\begin{split}
   \forall \gamma>0,\forall s\ge 0, \quad  \left\| e^{s \Lc_\gamma} \right\|_{\infty} \le 1 \quad {\rm{or}}
   \quad  \forall \gamma>0, \forall s\ge 0,  \quad \left\| e^{s \Lc_\gamma^*} \right\|_{\infty} \le 1.
\end{split}   
\end{equation}

\item {\rm{(Centering property)}} We have that $\Pc \Lc^{(1)} \Pc = 0$. Here $\mathcal P$ is the projector associated to $  \Lc^{(2)}$ defined by
\begin{align}
	\label{eq:def_Pi}
	\Pc := \lim_{t \rightarrow \infty} e^{t \Lc^{(2)}}.\\
\end{align}
It is well-defined thanks to i) and ii) as proved below. 
\end{enumerate}
\medskip
Under such assumptions, we have for any $t>0$:
\begin{equation*}
\lim_{\gamma \rightarrow \infty} e^{t \Lc_\gamma}=\Pc e^{t \Lc_\infty} \Pc
\end{equation*}
with
\begin{align*}
\Lc_\infty
& = \Pc \Lc^{(0)} \Pc 
  - \Pc\Lc^{(1)}
    \left( \Lc^{(2)} \right)^{-1}
    \Lc^{(1)} \Pc.
\end{align*}
Here $\left( \Lc^{(2)} \right)^{-1}$ is the pseudo-inverse of $\Lc^{(2)}$. 
\end{thm}
\begin{rmk}
\
\begin{enumerate}[1.]
\item The pseudo-inverse can be defined as
\begin{align}
\label{def:pseudo-inverse}
  \left( \Lc^{(2)} \right)^{-1} := -\int_0^\infty \left( e^{s \Lc^{(2)}} - \mathcal P \right) ds \ .
\end{align}
It is taken as zero on $\Ker \Lc^{(2)} = \Im \Pc$ and the inverse of $\Lc^{(2)}$ upon restricting to $\Ker \Pc$. The existence of this integral results from properties \eqref{it:spec_gap} and \eqref{it:ass_erg} established in Step $0$ of the proof.
\item As explained in the Step $1$ of the proof the projector $\Pc$ defined by Eq.~\eqref{eq:def_Pi} is the projector onto $\Ker \Lc^{(2)}$ parallel to $\Im \Lc^{(2)}=\left( \Ker (\Lc^{(2)})^* \right)^\perp$. An alternative formula is the following. The conditions i) and ii) imply we can find a basis $(r_i)_{i=1}^p$ of  $\Ker \Lc^{(2)}$ and a basis $(\ell_i)_{i=1}^p$ of $ \Ker (\Lc^{(2)})^*$ which are biorthogonal, i.e. $\langle \ell_i, r_j\rangle=\delta_{i,j}$ such that $\Pc (x)=\sum_{i=1}^k \langle \ell_i , x \rangle  r_i$ for any $x \in V$. 
\item While the theorem is formulated as a strong perturbation one, using the generalization mentioned in Subsection \ref{sec:further remarks} where $\Lc^{(0)} = \Lc^{(0)}_\gamma$ is allowed to depend on $\gamma$, it can be reformulated as a weak perturbation one in the appropriate time scale. If $\varepsilon\mapsto \Lc(\varepsilon)$ is a two times differentiable function in $0$ taking value in the set of endomorphisms on $V$. 

Setting $\Lc^{(2)}=\Lc (0)$, $\Lc^{(1)}=\Lc^{\prime} (0)$ and $\Lc^{(0)}=\frac{1}{2} \Lc^{\prime\prime} (0)$ and assuming the $\Lc^{(i)}$'s satisfy the four assumptions of the theorem, using $\gamma = \frac{1}{\varepsilon}$, one easily gets when $t>0$
$$
   \lim_{\varepsilon\to 0}e^{t\varepsilon^{-2}\Lc (\varepsilon)}=\Pc e^{t\Lc_\infty}\Pc.
$$

With the same framework, setting $\Lc^{(2)}=\Lc (0)$, $\Lc^{(1)}=0$ and $\Lc^{(0)}= \Lc^{\prime} (0)$, using $\gamma = \frac{1}{\sqrt{\varepsilon}}$, one gets for $t>0$
$$ \lim_{\varepsilon\to 0}e^{t\varepsilon^{-1}\Lc (\varepsilon)}=\Pc e^{t\Lc'(0)}\Pc $$

Furthermore, by continuity,
$$ \lim_{\varepsilon\to 0}e^{t\Lc(\varepsilon)}=e^{t\Lc(0)} \ .$$

\item 
Such limiting generators often appear in the context of homogenization theory \cite{pavliotis2008multiscale,Gardiner09}, and can take various equivalent formulations. On the physics side, the associated dimension reduction is also very common. Restricting to the quantum context, let us mention \cite{Cohen-Tanoudji92, GardinerZoller04, BoutenSilberfarb2008, BoutenHandelSilberfarb2008, Walls-Milburn08, BauerBernardJin17, BoutenGough2018}.
\end{enumerate}
\end{rmk}

\begin{proof}
The Hilbert space structure on $V$ induces a norm which is denoted $\Vert\cdot \Vert$. We also write $\Vert\cdot \Vert$ for the operator norm on the algebra of endomorphisms on $V$.

\bigskip

{\bf {Step 0: Spectral consequences of the spectral and contractivity properties i) and ii)}} 

We will prove in this step that the following properties hold:
\begin{itemize}
\item Real Spectral Gap \eqref{it:spec_gap} property:
\begin{equation} 
 \Re \Spec \Lc^{(2)} \subset (-\infty, 0] \; \text{and} \; \lambda=0 \; \text{is the only eigenvalue of}\;  \Lc^{(2)} \; \text{with}\;  \Re \lambda = 0.
\label{it:spec_gap}
\tag{RSG}
\end{equation}
\item {{Semi-simplicity}} \eqref{it:ass_erg} property\footnote{This property is equivalent to one of the following assertions:
\begin{itemize}
\item the restriction to $\Ker\left[ \left( \Lc^{(2)} \right)^m \right]$ is diagonalizable, with $m \in \N$ being the algebraic multiplicity. 
\item the algebraic multiplicity $m$ of $\lambda=0$ matches the geometric multiplicity $\dim \Ker \Lc^{(2)}$.
\end{itemize}}:
\begin{equation}
\label{it:ass_erg} \tag{SS}
\text{The eigenspace of  $\Lc^{(2)}$ associated to $0$ has only one dimensional Jordan blocks.} 
\end{equation}  
\end{itemize}

Let us start with \eqref{it:spec_gap}. The spectra of $\Lc^{(2)}$ and ${\Lc^{(2)}}^*$ are complex conjugate. Hence, without loss of generality, we can assume that ii) holds with $\Lc^{(2)}$, i.e. the latter generates a contracting semigroup. Suppose $\Lc^{(2)}$ has an eigenvalue $\lambda$ with strictly positive real part and denote $v_\lambda$ a corresponding eigenvector. Then $e^{t\Lc^{(2)}}v_\lambda=e^{t\lambda}v_\lambda$. It follows that $\|e^{t\Lc^{(2)}}v_\lambda\|_\infty=e^{t\Re\lambda}\|v_\lambda\|_\infty>\|v_\lambda\|_\infty$ and therefore $e^{t\Lc^{(2)}}$ is not a contraction with respect to $\Vert \cdot \Vert_{\infty}$. By contradiction we have $\Re \Spec \Lc^{(2)}\subset (-\infty,0]$. Since we assumed furthermore in i)  that there are no purely imaginary eigenvalues, \eqref{it:spec_gap} follows.

We now prove \eqref{it:ass_erg}. Let us first observe that by replacing in Eq. \eqref{eq:provisoire} $s$ by $s/{\gamma}^2$ and sending then $\gamma \to \infty$ we get that $\sup_{s\ge 0} \Vert e^{s\Lc^{(2)}}\Vert_{\infty}\le 1$ or $\sup_{s\ge 0} \Vert e^{s{\Lc^{(2)}}^*} \Vert_{\infty}\le 1$. Assume now by contradiction that \eqref{it:ass_erg} does not hold. Then there exists a vector $v_0$ such that $e^{t\Lc^{(2)}}v_0=\operatorname{Poly}(t)$ where $\operatorname{Poly}$ is a non zero vector with polynomial coefficients (in any given basis). Hence, its norm (any one) grows as $t$ goes to infinity. It follows that $e^{t\Lc^{(2)}}$ is not a contraction. A similar argument shows that $e^{t{\Lc^{(2)}}^*}$ cannot be a contraction either. Hence the existence of a non trivial Jordan block associated to the eigenvalue $0$ contradicts the contractivity assumption ii)  made on $\Lc^{(2)}$ or ${\Lc^{(2)}}^*$.
 
\bigskip

Before we start the proof, which will be divided in six steps, let us observe that, by the equivalence of the norms and the isometry property with respect to the norm $\Vert\cdot \Vert$,  we obtain that there exists $C>0$ such that
\begin{equation}
\label{eq:provisoire2}
 \sup_{\gamma>0 } \sup_{s\ge 0} \left\| e^{s \Lc_\gamma} \right\|  \le C, \quad  \sup_{\gamma>0 } \sup_{s\ge 0} \left\| e^{s \Lc_\gamma^*} \right\|  \le C.
\end{equation}

\bigskip 

{\bf Step 1: Existence and characterization of the projector $\Pc$ defined in Eq.~\eqref{eq:def_Pi}}

We will prove here the existence of $\Pc$ by using a Jordan normal form decomposition of $\Lc^{(2)}$.

Indeed considering the minimal polynomial $\pi(X)=\prod_{j=1}^\ell (X-\lambda_j)^{\nu_j}$ of $\Lc^{(2)}$, we have the decomposition into $\Lc^{(2)}$-invariant subspaces:  $V =\oplus_{j=1}^\ell {\Ker}\, (\Lc^{(2)} -\lambda_j)^{\nu_j}$. Here the $\lambda_j$ are the distinct eigenvalues of  $\Lc^{(2)}$ with $\lambda_1=0$ and for $j\ge 2$, $\Re \lambda_j <0$ (by \eqref{it:spec_gap}). Since $\nu_1=1$, by \eqref{it:ass_erg}, it follows immediately that ${\mathcal P}$ is well defined. Moreover, we have trivially that for any $v\in \Ker  \Lc^{(2)}$, $\Pc(v) =v \in \Im \Pc$ on the one hand and $\Im \Lc^{(2)} \subset \Ker \Pc$ on the other hand. The fact that only Jordan blocks of size $1$ associated to the zero eigenvalue are allowed implies that $\Im \Lc^{(2)} \cap \Ker \Lc^{(2)} =\{ 0\}$, and by the rank formula, $V=\Im \Lc^{(2)} \oplus \Ker \Lc^{(2)}=\Ker \mathcal P\oplus \Im \mathcal P$. Hence $\Pc$ is the projector on $\Ker \Lc^{(2)}$ parallel to $\Im \Lc^{(2)}$. In the rest of the proof we denote $\mathcal P_\perp=\id-\mathcal P$.

\bigskip

{\bf Step 2: Block decomposition of the semigroup}

Let us study the evolution of the exponential map $t \mapsto e^{t \Lc_\gamma}$. It is the unique solution to the Cauchy problem $\partial_t {{\mathcal K}}_t^\gamma = \Lc_\gamma {\mathcal K}_t^\gamma$ with initial condition ${\mathcal K}_0^\gamma=\Id$.

We remark that $\Lc_\gamma \stackrel{\gamma \rightarrow \infty}{\sim} \gamma^2 \, \Lc^{(2)}$. Hence it is natural to decompose ${\mathcal K}^\gamma$ according to the direct sum ${\rm{Im}} (\Pc) \oplus {\rm{Im}} (\Pc_\perp)$. This is given by
\begin{equation*}
e^{t \Lc_\gamma} = {\mathcal G}^\gamma_t + {\mathcal H}^\gamma_t + {\mathcal I}^\gamma_t + {\mathcal J}^\gamma_t,
\end{equation*}
with 
\begin{equation}\label{eq:def_blocs}
\begin{split}
&      {\mathcal G}^\gamma_t := \Pc e^{t \Lc_\gamma} \Pc,
\quad {\mathcal H}^\gamma_t := \Pc_\perp e^{t \Lc_\gamma} \Pc,\\
&{\mathcal I}^\gamma_t := \Pc e^{t \Lc_\gamma} \Pc_\perp,
\quad {\mathcal J}^\gamma_t := \Pc_\perp e^{t \Lc_\gamma} \Pc_\perp \ .
\end{split}
\end{equation}

\bigskip

{\bf Step 3: $\| {\mathcal H}_t^\gamma\|$, $\| {\mathcal I}_t^\gamma\|$ and $\| {\mathcal J}_t^\gamma\|$ asymptotically vanish as $\gamma\to \infty$.}

For any $t>0$, $\| {\mathcal H}_t^\gamma\|$, $\| {\mathcal I}_t^\gamma\|$ and $\| {\mathcal J}_t^\gamma\|$ will go to $0$ as $\gamma\to \infty$ as a consequence of the fact that $\Pc$ and $\Pc_\perp$ are bounded and the fact that there exists $C>0$ and $g>0$ such that for any $\gamma>0$ and $t\in \R_+$,
\begin{align}
\label{eq:toto}
        \|\Pc_\perp e^{t\Lc_\gamma}\|
 \leq & \ C \left( e^{-tg\gamma^2} + \tfrac{1}{\gamma} \right) \ ,
\end{align}
and
\begin{align}
\label{eq:toto2}
       \| e^{t\Lc_\gamma}\Pc_\perp\|
\leq & \ C \left( e^{-tg\gamma^2} +  \tfrac{1}{\gamma} \right) \ .
\end{align}
First of all, the latter Eq.~\eqref{eq:toto2} is a consequence of the former Eq.~\eqref{eq:toto} by considering the adjoint. Indeed, because taking the dual is an isometric operation, we have $\| e^{t\Lc_\gamma}\Pc_\perp\|=\| \Pc_\perp^* e^{t\Lc_\gamma^*}\|$. As such, we only need to see that the pair $( \Pc_\perp^*, \Lc_\gamma^*)$ falls into the same setting as $(\Pc_\perp, \Lc_\gamma)$. This is obvious because $\left( \Lc^{(2)} \right)^*$ satisfies the same assumptions as $\Lc^{(2)}$, $\Lc_\gamma^*$ has a similar decomposition to $\Lc_\gamma$ and by the definition of $\Pc$ in Eq.~\eqref{eq:def_Pi}, we have the expression:
$$ \Pc^* = \lim_{t \rightarrow \infty} e^{t\left( \Lc^{(2)} \right)^*} \ .
$$

Now, we can focus on proving Eq.~\eqref{eq:toto}. As we will see the centering assumption is not used in this proof. We start by explaining why there exists $C>0$ and $g>0$ such that:
\begin{align}
\label{eq:ineq_pi_perp_Lnd}
\forall t \in \R_+, \quad 
\|\Pc_\perp e^{t\Lc^{(2)}}\|
\leq & \ C e^{-tg} \ .
\end{align}
Thanks to the  semi-simplicity property \eqref{it:ass_erg}, the Jordan-Chevalley decomposition of $\Lc^{(2)}$ gives that
$$ \forall t \in \R_+, \quad 
\|\Pc_\perp e^{t\Lc^{(2)}}\|
\leq \left| P(t) \right|
     e^{t\Lambda} $$
where $P$ is a polynomial and $\Lambda = \max\{ \Re \lambda \ | \ \lambda \neq 0, \ \lambda \in \Spec \Lc^{(2)} \}$. On the other hand the real spectral gap property \eqref{it:spec_gap} implies $\Lambda < 0$. The decreasing exponential absorbs the polynomial and therefore Eq.~\eqref{eq:ineq_pi_perp_Lnd} holds with an appropriate choice of constants. Then, by applying Duhamel's principle to the ordinary differential equation (ODE)
$$ \partial_t e^{t \Lc_\gamma}
   =
   \gamma^2 \Lc^{(2)} e^{t \Lc_\gamma}
   +
   \left( \gamma \Lc^{(1)} + \Lc^{(0)} \right) e^{t \Lc_\gamma} \ ,
$$
we have:
$$ e^{t\Lc_\gamma}
 = e^{t \gamma^2 \Lc^{(2)}}
 + \int_0^t ds \ e^{(t-s) \gamma^2 \Lc^{(2)}}
            \left( \gamma \Lc^{(1)} + \Lc^{(0)} \right) e^{s\Lc_\gamma} \ .
$$
Therefore, with possibly different constants $C>0$, the following bounds hold:
\begin{align*}
     \|\Pc_\perp e^{t\Lc_\gamma}\| 
\leq & \ \| \Pc_\perp e^{t \gamma^2 \Lc^{(2)}}\| 
       + \left\|\ \int_0^t ds \ \Pc_\perp e^{(t-s) \gamma^2 \Lc^{(2)}}
            \left( \gamma \Lc^{(1)} + \Lc^{(0)} \right) e^{s\Lc_\gamma}
         \right\| \\
\leq & \ C e^{-tg\gamma^2}
       + C \int_0^t ds \ e^{-(t-s)g\gamma^2}
                    \left( \gamma \|\Lc^{(1)}\|  + \|\Lc^{(0)}\|\right)
					\| e^{s\Lc_\gamma}\| \\          
\leq & \ C e^{-tg\gamma^2}
       + C \left( \gamma^{-1} + \gamma^{-2} \right)
           \sup_{0 \leq s \leq t}\| e^{s\Lc_\gamma}\| 
           \ .
\end{align*}
With the uniform boundedness property given by Eq.~\eqref{eq:provisoire2} we are done.

\bigskip

{\bf Step 4: The limit of ${\mathcal G}^\gamma_t$ as $\gamma \to \infty$.}

As a very general fact, it turns out that the pair $({\mathcal G}^\gamma, {\mathcal H}^\gamma)$ follows a closed evolution equation. The same holds for $({\mathcal I}^\gamma, {\mathcal J}^\gamma)$. Since we are interested only in the former we do not discuss the latter. By using $\Pc^2=\Pc$ and $(\Pc_\perp)^2= \Pc_\perp$, we have that  
\begin{equation}
\label{eq:GH}
\begin{cases}
\partial_t {\mathcal G}_t^\gamma =(\Pc \Lc_\gamma \Pc) {{\mathcal G}}_t^\gamma + (\Pc \Lc_\gamma \Pc_\perp) {\mathcal H}_t^\gamma, \quad {\mathcal G}_0^\gamma = \Pc, \\
\partial_t {\mathcal H}_t^\gamma =(\Pc_\perp \Lc_\gamma \Pc) {{\mathcal G}}_t^\gamma + (\Pc_\perp \Lc_\gamma \Pc_\perp) {\mathcal H}_t^\gamma, \quad {\mathcal H}_0^\gamma =  {\bf 0}.
 \end{cases}
\end{equation}
Solving the second equation in terms of ${{\mathcal G}}^\gamma$ and using again that $\Pc_\perp^2=\Pc_\perp$ and $\Pc_\perp  {\mathcal H}^\gamma={\mathcal H}^\gamma$, we get that
\begin{align*}
{\mathcal H}_t^\gamma &= \int_0^t e^{(t-s) \Lc_\gamma^\perp} \, \Pc_\perp \Lc_\gamma \Pc \, {{\mathcal G}}_s^\gamma \, ds=\int_0^t \left[ \Pc_{\perp}  e^{(t-s) \Lc_\gamma^\perp} \, \Pc_\perp \right]\; \left[ \Pc_\perp \Lc_\gamma \Pc\right]  \, {{\mathcal G}}_s^\gamma \, ds
\end{align*}
where 
$$\Lc_\gamma^\perp := \Pc_\perp\Lc_\gamma\Pc_\perp.$$
We deduce then the following closed equation for ${\mathcal G}^\gamma$:
\begin{align*}
  \partial_t {\mathcal G}_t^\gamma 
= (\Pc \Lc_\gamma \Pc) {{\mathcal G}}_t^\gamma + (\Pc \Lc_\gamma \Pc_\perp)\int_0^t \Big[ \Pc_{\perp}  e^{(t-s) \Lc_\gamma^\perp} \, \Pc_\perp \Big]\; \left[ \Pc_\perp \Lc_\gamma \Pc\right]  \, {{\mathcal G}}_s^\gamma \, ds\ .
\end{align*}
This differential formulation is turned into an integral form, and via a Fubini argument, we get
\begin{equation}
\label{eq:closedG}
{\mathcal G}_t^\gamma = \Pc + \int_0^t ds \, {{\mathcal M}}_\gamma {{\mathcal G}}_s^\gamma\; +\; {{\mathcal R}}_t^\gamma
\end{equation}
with the operator ${\mathcal M}_\gamma \in \End\left( V  \right)$ defined by
\begin{equation*}
\begin{split}
{{\mathcal M}}_\gamma &:= (\Pc \Lc_\gamma \Pc)  - (\Pc \Lc_\gamma \Pc_\perp) \,  (\Lc_\gamma^\perp)^{-1} \, (\Pc_\perp \Lc_\gamma \Pc) \\
\end{split}
\end{equation*}
and the remainder term given by
\begin{equation}
\label{eq:remainderRperp}
\Rc_t^\gamma =(\Pc\Lc_\gamma \Pc_\perp) (\Lc_\gamma^\perp)^{-1} \int_0^t \left[ \Pc_{\perp}  e^{(t-s) \Lc_\gamma^\perp} \, \Pc_\perp \right]\; \left[ \Pc_\perp \Lc_\gamma \Pc\right]  \, {{\mathcal G}}_s^\gamma \, ds.
\end{equation}

We prove later that there exists $\gamma_0>0$ and $C_0>0$ such that for $\gamma > \gamma_0$,  
\begin{equation}
\label{eq:estimate-n}
\sup_{t >0} \| \Rc_t^\gamma \| \le \tfrac{C_0}{\gamma^{2}}, \quad \| {{\mathcal M}}_\gamma -\Lc_{\infty}\| \le \tfrac{C_0}{\gamma}.
\end{equation}
Let us now finish the proof modulo this claim. In order to do so, we rewrite Eq.~\eqref{eq:closedG} as
\begin{equation}
\label{eq:GGG}
{\mathcal G}_t^\gamma = \Pc + \int_0^t  \Lc_\infty {{\mathcal G}}_s^\gamma\, ds \; +\; {\tilde{{\mathcal R}}_t^\gamma}
\end{equation}
with
\begin{equation}
\label{eq:RRR}
{\tilde{{\mathcal R}}}_t^\gamma = \int_0^t [{{\mathcal M}}_\gamma -\Lc_\infty] {\mathcal G}_s^\gamma \, ds \; +\; {{\mathcal R}}_t^\gamma.
\end{equation}
Then, defining ${\mathcal Y}^\gamma_t:={\mathcal G}_t^\gamma - e^{ t \Lc_\infty} \Pc$, since ${\mathcal G}_0^\gamma=\Pc$ (see Eq.~\eqref{eq:GH}), we have 
 \begin{equation*}
\partial_t {\mathcal Y}^\gamma= \Lc_{\infty} {\mathcal Y}^\gamma +\partial_t {\tilde {\mathcal R}^{\gamma}}, \quad {\mathcal Y}_0^\gamma=0.
\end{equation*}
Solving this ODE and performing an integration by parts, we get that
 \begin{equation*}
 {\mathcal G}_t^\gamma - e^{ t \Lc_\infty} \Pc ={\mathcal Y}_t^\gamma=  \Lc_{\infty} \int_0^t e^{(t-s)\Lc_\infty} {\tilde{{\mathcal R}}}_s^\gamma \, ds \; +\; {\tilde{\mathcal R}}_t^\gamma.
\end{equation*}
 Hence:
\begin{equation*}
    \| {\mathcal G}_t^\gamma - e^{ t \Lc_\infty} \Pc \|
\le e^{t \| \Lc_\infty\|} \sup_{0\le s \le t} \| {\tilde {\mathcal R}}_s^\gamma \| \ .
\end{equation*}
Since by the uniform boundedness property \eqref{eq:provisoire2}, $\sup_{s \ge 0} \| e^{s \Lc_\gamma} \| \leq C$, we have with possibly different constants $C>0$, that $\| \Pc \| \le C$ by Eq.~\eqref{eq:def_Pi}. Consequently, by the definition of ${\mathcal G}^\gamma$ in \eqref{eq:def_blocs}, that 
\begin{equation}
\label{eq:boundGGG}
\sup_{s \ge 0} \| {\mathcal G}_s^\gamma \| \le C.
\end{equation}
Hence by Eq.~\eqref{eq:RRR} and Eq.~\eqref{eq:estimate-n} we deduce that for $\gamma >\gamma_0$, 
\begin{equation*}
\sup_{0\le s \le t} \| {\tilde {\mathcal R}}_s^\gamma \| \le C_0 \left( \tfrac{C t} {\gamma} + \tfrac{1}{\gamma^{2}} \right).
\end{equation*}
We conclude that $\|{\mathcal G}_t^\gamma - e^{ t \Lc_\infty} \Pc\| \to 0$ as $\gamma \to \infty$ and this concludes the proof of the theorem. 

\bigskip

{\bf Step 5: Some spectral perturbation theory.} Before proving the remaining claim \eqref{eq:estimate-n}, let us show that there exists $K>0$  and $g>0$ such that for any $\gamma$ large enough,
\begin{equation}
\label{eq:bound_exp_Lperp-1}
\|\Pc_\perp e^{t\Lc_{\gamma}^\perp} \Pc_{\perp}\| \leq K e^{-t\gamma^2 g} \ .
\end{equation}

For convenience, for any operator $T \in \End(V)$, we define $\widetilde{T} \in \End( \Pc_\perp V )$ as the restriction of $\Pc_\perp T \Pc_\perp$ to $\Pc_\perp V$. Observe first that
\begin{equation*}
\frac{1}{\gamma^2} \widetilde{\Lc_\gamma}
 = 
 \widetilde{ \Lc^{(2)} }
 +
 \frac{1}{\gamma} \widetilde{\Lc^{(1)}}
 +
 \frac{1}{\gamma^2} \widetilde{ \Lc^{(0)} } \ .
\end{equation*}
Thanks to the real spectral gap property \eqref{it:spec_gap}, the endomorphism $\widetilde{\Lc^{(2)}}$ has only eigenvalues with strictly negative real parts. Standard perturbation theory in finite dimension \cite[$\mathsection$ 2.5.1]{kato2013perturbation} states that eigenvalues change continuously as $\frac{1}{\gamma} \rightarrow 0$. As such, that for $\gamma$ large enough, there exists some $g>0$ such that
\begin{equation}
\label{eq:ineq_L_perp-1}
\max\left\{
  \Re(\lambda) \ | \ 
  \lambda \in \Spec\left( \frac{1}{\gamma^2} \widetilde{\Lc_\gamma} \right) \ 
  \right\}
  \leq -2 g
\end{equation}

In order to control $e^{t \Lc_\gamma^\perp}$, one cannot invoke the Jordan-Chevalley decomposition as eigenprojectors can behave poorly -- see \cite[$\mathsection$ 2.5.3]{kato2013perturbation} for example. Instead, we invoke the holomorphic functional calculus:
$$ f(T) = \frac{1}{2\pi i} \int_{\Cc} f(z) \left( z-T \right)^{-1} dz \ ,$$
which holds for any operator $T \in \End(V)$, any holomorphic $f: \C \rightarrow \C$ and any contour $\Cc$ encircling $\Spec T$. Since $\frac{1}{\gamma^2} \widetilde{\Lc_\gamma}$ has negative eigenvalues bounded away from $i\R$, there exists a constant $C>0$ such that for $\gamma>0$ large enough:
$$ \Spec\left( \frac{1}{\gamma^2} \widetilde{\Lc_\gamma} \right)
   \subset [-C, -g] \times [-C, C] \ .
$$
Using the boundary of this box as contour with parametrization $\left( z(s) \ ; 0 \leq s \leq 1 \right)$, the functional calculus yields for $T = \frac{1}{\gamma^2} \widetilde{\Lc_\gamma}$ and $f(z) = e^{t \gamma^2 z}$:
$$ \exp\left( t \tilde{\Lc_{\gamma}} \right)
   =
   \frac{1}{2\pi i} \int_{\Cc} e^{t \gamma^2 z} \left( z-\frac{1}{\gamma^2} \widetilde{\Lc_\gamma} \right)^{-1} dz \ .   
$$
Therefore:
\begin{align*}
       \|\Pc_\perp e^{t\Lc_{\gamma}^\perp} \Pc_{\perp}\|
=    & \ \| \exp\left( t \tilde{\Lc_{\gamma}} \right) \Pc_{\perp} \| \\
\leq & \ \| \exp\left( t \tilde{\Lc_{\gamma}} \right) \| \ \| \Pc_{\perp} \| \\
=    & \ \left\| \frac{1}{2\pi i} \int_{\Cc} e^{t \gamma^2 z} \left( z-\frac{1}{\gamma^2} \widetilde{\Lc_\gamma} \right)^{-1} dz \right\| \ \| \Pc_{\perp} \| \\
\leq & \ \| \Pc_{\perp} \| \ \frac{e^{-t \gamma^2 g}}{2\pi } \int_{[0,1]} \left\| \left( z(s) - \frac{1}{\gamma^2} \widetilde{\Lc_\gamma} \right)^{-1} \right\| \ |z'(s)| \ ds \ .
\end{align*}
 
 As long as the contour stays away from the spectrum, the resolvant $\left( z-\frac{1}{\gamma^2} \widetilde{\Lc_\gamma} \right)^{-1}$ is continuous in $z$ and in the operator $\frac{1}{\gamma^2} \widetilde{\Lc_\gamma} \stackrel{\gamma \rightarrow \infty}{\longrightarrow} \widetilde{\Lc^{(2)}}$. Therefore, the integrand is bounded by some constant $L$. We have thus proved Eq.~\eqref{eq:bound_exp_Lperp-1} with $K = \frac{\| \Pc_{\perp} \| L}{2 \pi}$.

\bigskip

{\bf Step 6: Proof of claim \eqref{eq:estimate-n}.}

Thanks again to the continuity of eigenvalues, for $\gamma$ large enough $\Ker \Lc_\gamma^\perp= \Pc V$ and $\Lc_\gamma^\perp$ restricted to $\Pc_\perp V$ is invertible. Using the definition of pseudo-inverse in Eq.~\eqref{def:pseudo-inverse}, we have that for $\gamma$ large enough, 
$$ \Pc_\perp (\Lc_\gamma^\perp)^{-1}\Pc_\perp
 = -\int_0^\infty \Pc_{\perp}  e^{t\Lc_\gamma^\perp}\Pc_\perp dt
$$
and thus, by Eq.~\eqref{eq:bound_exp_Lperp-1},
\begin{align*}
\| \Pc_\perp  (\Lc_\gamma^\perp)^{-1}\Pc_\perp\|\leq \frac{C}{\gamma^2} \ .
\end{align*}
By using the definition \eqref{eq:def_Pi} for $\mathcal P$ we have $\Pc \Lc^{(2)}=0=\Lc^{(2)} \Pc$ so that  
$$\|\Pc \Lc_\gamma\Pc_\perp\|=\|\Pc(\Lc^{(0)}+\gamma\Lc^{(1)})\Pc_\perp\| \leq C (\gamma+1) \ ,$$
and similarly 
$$\|\Pc_\perp \Lc_\gamma\Pc\|= \|{\Pc_\perp}(\Lc^{(0)}+\gamma\Lc^{(1)})\Pc\|   \leq C (\gamma+1) \ .$$ 
Hence, starting from Eq.~\eqref{eq:remainderRperp} and using $\sup_{s\ge 0} \|{\mathcal G}_s^\gamma\|\leq C$ (see Eq.~\eqref{eq:boundGGG}) in combination with the three previous inequalities, we have:
$$\|{\mathcal R}_t^\gamma\|\leq C  \int_0^t\|\Pc_\perp e^{(t-s)\Lc_\gamma^\perp} \Pc_\perp\|ds.$$
Therefore, invoking again  Eq.~\eqref{eq:bound_exp_Lperp-1}, we obtain the first part of the claim: for $\gamma$ large enough, we have that
\begin{equation}
\label{eq:bound_I_t-1}
\sup_{t\ge 0} \|{\mathcal R}_t^\gamma\|\leq \frac{C}{\gamma^{2}} \ .
\end{equation}

Now we prove the second part of the claim \eqref{eq:estimate-n}. Let us denote by $(\Lc^{(i)})^{\perp}$ the restriction of $\Pc_\perp \Lc^{(i)} \Pc_\perp$ to  $\Pc_\perp V$. Invoking the fact that 
$$\Pc \Lc^{(2)} =\Lc^{(2)}\Pc=0$$ 
and the centering property iii) we obtain
\begin{equation*}
 \Pc \Lc_\gamma \Pc = \Pc \Lc^{(0)} \Pc \ ,
\end{equation*}
and
\begin{equation*}
\begin{split}
&\Pc \Lc_\gamma \Pc_\perp=\gamma \Pc \Lc^{(1)} \Pc_{\perp} + \Pc \Lc^{(0)} \Pc_{\perp} \ ,\\ 
&\Pc_\perp \Lc_\gamma \Pc =\gamma \Pc_\perp \Lc^{(1)} \Pc + \Pc_{\perp} \Lc^{(0)} \Pc \ .
\end{split}
\end{equation*}
Moreover, since $[\Lc^{(2)}]^\perp $ is invertible, we have
\begin{align*}
    \ (\Lc_\gamma^\perp)^{-1}
= & \ \frac{1}{\gamma^2} \left( [\Lc^{(2)}]^{\perp}\right)^{-1}
    \ \left( {\rm{Id}} + \tfrac{1}{\gamma}   [\Lc^{(1)}]^{\perp} \left( [\Lc^{(2)}]^{\perp}\right)^{-1} + \tfrac{1}{\gamma^2} [\Lc^{(0)}]^{\perp}  \left( [\Lc^{(2)}]^{\perp}\right)^{-1} \right)^{-1} \ .
\end{align*}
Plugging all this in the expression
\begin{equation*}
\begin{split}
{\mathcal M}_\gamma&= (\Pc \Lc_\gamma \Pc)  - (\Pc \Lc_\gamma \Pc_\perp) \,  (\Lc_\gamma^\perp)^{-1} \, (\Pc_\perp \Lc_\gamma \Pc)
\end{split}
\end{equation*}
and since the operators $\Big([\Lc^{(2)}]^\perp\Big)^{-1}$,  $(\Lc^{(1)})^{\perp}$ and  $(\Lc^{(0)})^{\perp}$ are uniformly bounded in $\gamma$, we obtain
\begin{equation*}
\|{{\mathcal M}}_\gamma-\Lc_\infty\|\leq \frac{C}{{\gamma}} \ .
\end{equation*}
Claim \eqref{eq:estimate-n} has been proven.
\end{proof}

\begin{rmk}
It is clear from the proof given above that the conclusion of Theorem~\ref{thm:homogenization} holds if the assumptions settled there are replaced by the real spectral property \eqref{it:spec_gap}, the semi-simplicity property \eqref{it:ass_erg}, the uniform boundedness property \eqref{eq:provisoire2} and the centering property iii). 
\end{rmk}

\subsection{Application to general Lindbladians}

As we shall now see, semigroups generated by Lindbladians satisfy for free the contractivity condition ii) of Theorem~\ref{thm:homogenization}. Consequently we easily obtain the following corollary.  

\begin{corollary}
\label{cor:homo Lindbladian}
Consider a family of Lindbladians $(\Lc_{\gamma})_{\gamma>0}$, $\Lc^{(0)}, \Lc^{(1)}, \Lc^{(2)}$ satisfying the decomposition form given by Eq. \eqref{eq:perturbed_lind} such that  $\Lc^{(2)}$ satisfies the spectral property i) and $\Lc^{(1)}$ the centering condition iii) of Theorem~\ref{thm:homogenization}, then the conclusion of Theorem~\ref{thm:homogenization} holds. 
\end{corollary}
\begin{proof}
We only have to check the contractivity property ii) to apply Theorem~\ref{thm:homogenization} (with $V=M_d (\mathbb C)$ equipped with the Hilbert--Schmidt inner product). Any Lindbladian $\Lc$ is such that its dual $\Lc^*$ is the generator of a semigroup of identity preserving positive maps. Defining $\|\cdot\|_\infty$ as the operator $2$-norm\footnote{The operator $2$-norm of a matrix $M\in M_d (\C)$ is the operator norm of $M$ considered as an operator on $\C^d$, the latter being equipped with its standard Hilbert structure.} over matrices, the normed space $(M_d(\C),\|\cdot\|_\infty)$ is a $C^*$-algebra. Then choosing $\|\cdot\|_\infty$ to be the induced operator norm on $\End\left( M_d(\C) \right)$, Russo--Dye Theorem implies $\|e^{t\Lc^*}\|_\infty=1$  (see \cite[Corollary 1]{russo1966note}). Hence the contractivity property ii) holds and we have proven the missing assumptions of Theorem~\ref{thm:homogenization} whose conclusion therefore holds.
\end{proof}

\begin{rmk}
	This corollary shows we can apply directly our abstract homogenization result of Theorem~\ref{thm:homogenization} to semigroups generated by Lindbladians. Thus, using the generalizations of Subsection~\ref{sec:further remarks}, we have a rigorous proof of the results of \cite{macieszczak2016towards} concerning class A metastable states. In the language of said article, $\Ker  \Lc^{(2)}$ is a set of class A metastable states and $\Lc_{\infty}$ is the generator of the  intermediary time scale dynamics among them.
\end{rmk}

\subsection{Application to Lindbladians with a decoherence assumption}
\label{section:lindblad_perturbation}

Here we specialize Corollary \ref{cor:homo Lindbladian} to the case where $\Ker \Lc^{(2)}$ is the space of matrices that are diagonal in the pointer basis. Such setting is implied by our working assumptions. In this case the limiting semigroup has the remarkable property to be closely related to the Markov process $\mathbf{x}$ of our Main Theorem~\ref{thm:main}. 

We start this subsection by a lemma on the structure of Lindbladians, which justifies and explains the QND Assumption \ref{ass:(Nd)}:

\begin{lemma}\label{lemma:lindblad_structure}
Let $\Lc$ be a Lindbladian from $M_d(\C)$ to itself with a GKSL decomposition given by \eqref{def:gksl}. Let $(e_i)_{i=1}^d$ be any orthonormal basis and $\Pi$ be the orthogonal projector onto diagonal matrices in this basis. Then:
\begin{enumerate}[1)]
\item If $H$ and every $L_k$ is diagonal in $(e_i)_{i=1}^d$ then 
the super-operator $\Lc: M_d(\C) \rightarrow M_d(\C)$ is diagonalizable with eigenvectors $E_{i,j}:=e_i e_j^*$, $1 \leq i,j \leq d$, and associated eigenvalues:
\begin{align}
\label{eq:L_eigenequation}
\tau_{i,j}  = \ & - \half \sum_{k=1}^\ell \left| (L_k)_{i,i} - (L_k)_{j,j} \right|^2 \\
                 & \quad \quad - \imath \left( H_{i,i} - H_{j,j}
                   + \sum_{k=1}^\ell \Im \left( \overline{(L_k)_{i,i}} (L_k)_{j,j} \right) \right)
                   \ \nonumber.
\end{align}
\item The $L_k$ are diagonal in $(e_i)_{i=1}^d$ iff $\Pi\Lc\Pi=0$ .\label{it:diag_diag_0_equiv_diag}
\item The $L_k$ and $H$ are diagonal in $(e_i)_{i=1}^d$ iff $\Pi\Lc=0=\Lc\Pi$ .\label{it:struct_nd} 
\end{enumerate}
Furthermore, for any Lindblad operator $\Lc$, $\Pi\Lc^*\Pi$\footnote{This operator is defined on the whole matrix space but we identify it with its restriction to diagonal matrices $\Pi\Lc^*\Pi: \Span_\C\{E_{i,i}: i=1,\dotsc,d\}\to \Span_\C\{E_{i,i}: i=1,\dotsc,d\}; X\mapsto \Pi\Lc^*\Pi(X)$.} is a generator of a Markov process on the pointer states $(E_{i,i})_{i=1}^d$.

\end{lemma}

Now, let us introduce a weaker form of Assumption \ref{ass:(Id)}.

Since we require the limiting dynamics $\xb$ to be concentrated on pointer states $(E_{i,i})_{i=1}^d$, it is natural to require that, in long time, the dominating dynamics generated by $\Lc^{(2)}$ projects onto $\Span_\C\{E_{i,i}: i=1,\dotsc, d\}$: 
$$ \lim_{t\to\infty}\Pi_{\perp}e^{t\Lc^{(2)}} = 0 \ .$$
This behavior of the dominating dynamics is called decoherence. Under QND Assumption \ref{ass:(Nd)}, decoherence is equivalent to $\Re\tau_{i,j}<0$ in Eq.~\eqref{eq:L_eigenequation} for $i\neq j$, which is in turn equivalent to the following condition.

\begin{assumption}[Decoherence Condition]
\label{ass:Dec}
For any $i\neq j$, there exists $k\in\{1,\dotsc,\ell_2\}$ such that $(L_k^{(2)})_{i,i} \neq (L_k^{(2)})_{j,j}$.
\end{assumption}

Remark that the Identifiability Assumption \ref{ass:(Id)} implies this Decoherence Condition.

From the last statement of last lemma, we expect that under the QND Assumption \ref{ass:(Nd)} and the Decoherence Condition in Assumption \ref{ass:Dec}, whenever $\Lc^{(1)}=0$, the limit average dynamics will be given by the generator $\Pi{\Lc^{(0)}}^*\Pi$ of a Markov process on the pointer states $(E_{i,i})_{i=1}^d$. Next proposition proves such a limit with a Markov generator $T$ corrected by the contribution due to $\Lc^{(1)}$.  Although the proposition is a trivial consequence of our Main Theorem \ref{thm:main}, it will be a crucial step to establish the latter.

\begin{proposition}[Convergence of the mean]
\label{proposition:Lindblad_perturbation}
Assume the QND Assumption \ref{ass:(Nd)} and the Decoherence Condition in Assumption \ref{ass:Dec}. Let $(e_i)_{i=1}^d$ be the pointer basis (defined in the QND Assumption \ref{ass:(Nd)})  and $\Pi$ be the orthogonal projector onto diagonal matrices in this basis.  Then we have for any $t>0$ that\footnote{In this formula we adopt the convention that the semigroup $e^{tT^*}$ acts on a diagonal matrix $\operatorname{diag} (x_1,\dotsc,x_d)$ by transforming it in the diagonal matrix $\operatorname{diag} (e^{tT^*}(x_1,\dotsc,x_d))$.} 
$$\lim_{\gamma\to\infty} e^{t\Lc_\gamma}=\Pi e^{tT^*} \Pi$$
where $T$ is the generator of the jump Markov process $\xb = (\xb_t)_{t\ge 0}$ on the pointer basis $(e_i)_{i=1}^d$ with transition rates given by Eq.~\eqref{eq:T-rates}.

In particular, if $\rho^\gamma$ is the solution of Eq. \eqref{eq:SDE} starting from $\varrho \in {\mathcal S}$ then its mean converges
\begin{equation}
\label{eq:mean-rho}
\lim_{\gamma \to \infty} \E_{\varrho} [\rho_t^\gamma] =\E_{\mu_{\varrho}} \left[ \xb_t \xb_t^* \right]
\end{equation}
where the initial probability measure $\mu_{\varrho}$ of $\xb = (\xb_t)_{t\ge 0}$ is defined by $\mu_{\varrho} (e_i) =\langle e_i, \varrho e_i \rangle$ for $i=1,\ldots, d$.  
\end{proposition}
Finally we conclude this section by giving the proofs of Lemma \ref{lemma:lindblad_structure} and Proposition \ref{proposition:Lindblad_perturbation}
\begin{proof}[Proof of Lemma \ref{lemma:lindblad_structure}]
First, we recall the GKSL decomposition of the Lindbladian $\Lc$ given in Eq.~\eqref{def:gksl}:
$$ \forall X \in M_d(\C), \quad \Lc(X) = -\imath [H,X] + \sum_{k=1}^\ell \left( L_k X L_k^* - \half \{L_k^* L_k, X \} \right) \ .$$
We are now ready to prove the lemma.

\bigskip

\noindent 1)  The eigenvalues of $\Lc$ given in Eq.~\eqref{eq:L_eigenequation} associated to the eigenvectors $E_{i,j}$ are proven via the following straightforward computation:
\begin{align*}
\Lc(E_{i,j}) = &- \imath[H,E_{i,j}] + \sum_{k=1}^\ell \left( L_k E_{i,j} L_k^* - \half \{L_k^* L_k, E_{i,j} \} \right)\\
= & - \imath \left( H_{i,i} - H_{j,j} \right) E_{i,j}+ \sum_{k=1}^\ell \left( {(L_k)_{i,i}} \overline{(L_k)_{j,j}} - \half \left| (L_k)_{i,i} \right|^2 - \half \left| (L_k)_{j,j} \right|^2  \right) E_{i,j} \ .
\end{align*}
This is indeed of the form ${\tau}_{i,j} E_{i,j}$ where
\begin{align*}
 {\tau}_{i,j}= &- \imath \left( H_{i,i} - H_{j,j} \right)
    + \sum_{k=1}^\ell \left( \imath \Im {(L_k)_{i,i}} \overline{(L_k)_{j,j}} - \half \left| (L_k)_{i,i} - (L_k)_{j,j} \right|^2  \right) \\
= & -\imath \left( H_{i,i} - H_{j,j} - \sum_{k=1}^\ell \Im {(L_k)_{i,i}} \overline{(L_k)_{j,j}} \right)
    - \half \sum_{k=1}^\ell \left| (L_k)_{i,i} - (L_k)_{j,j} \right|^2 \ .
\end{align*}

\bigskip

\noindent 2)
For all $i,j$ in $\{1, 2, \dots, d\}$, we have:
\begin{align*}
  &   \ \left( \Lc(E_{i,i}) \right)_{j,j} \\
  & = \ \langle e_j, \Lc (E_{i,i})  e_j \rangle \\
  & = \ - \imath \langle e_j, [H, E_{i,i}] e_j \rangle 
      + \sum_{k=1}^\ell \langle e_j, L_k E_{i,i} L_k^* e_j \rangle 
      - \half \langle e_j, \{ L_k^* L_k , E_{i,i} \} e_j \rangle \ .
\end{align*}
Now, because:
$$ \langle e_j, [H, E_{i,i}] e_j \rangle  = 0 \ , \quad
   \langle e_j, L_k E_{i,i} L_k^* e_j \rangle  = \left| \langle L_k e_i, e_j \rangle \right|^2 \ ,$$
$$ \langle e_j, \{ L_k^* L_k , E_{i,i} \} e_j \rangle = 2 \delta_{i,j} \left\Vert L_k e_j\right\Vert^2 \ ,$$
we obtain the formula:
\begin{align}
\label{eq:missing_formula}
      \left( \Lc(E_{i,i}) \right)_{j,j}
  & = \sum_{k=1}^\ell \left[ \left| \langle L_k e_i, e_j \rangle \right|^2 - \delta_{i,j} \left\Vert L_k e_j\right \Vert^2 \right] \ .
\end{align}
From formula \eqref{eq:missing_formula}, $\Pi\Lc\Pi=0$ iff for all $i,j,k,l$ in $\{1, 2, \dots, d\}$:
\begin{align}
\label{eq:piLpi_equivalence}
0 = & \ \left( (\Pi \Lc \Pi) (E_{i,k}) \right)_{j,l} \ = \ \left( \Lc(E_{i,i}) \right)_{j,j}\delta_{i,k}\delta_{j,l}\\
  = & \ \left(\sum_{k=1}^\ell \left| \langle L_k e_i, e_j\rangle \right|^2 - \delta_{i,j} \left\Vert L_k e_j\right\Vert^2 \right) \delta_{i,k}\delta_{j,l}
  \nonumber
  \ .
\end{align}

Picking $i \neq j$ implies $0 = \sum_{k=1}^\ell \left| \langle L_k e_i, e_j \rangle \right|^2$ and in consequence every $L_k$ is diagonal in $(e_i)_{i=1}^d$. Reciprocally, if every $L_k$ is diagonal in $(e_i)_{i=1}^d$, then Eq.~\eqref{eq:piLpi_equivalence} clearly holds.

\bigskip

\noindent 3) $\Pi\Lc=\Lc\Pi=0$ implies $\Pi \Lc \Pi=0$. Invoking the second statement of this lemma, this in turn implies that every $L_k$ is diagonal in $(e_i)_{i=1}^d$. Let us now prove that $H$ is also diagonal in this basis. Notice that the $L_k$'s being diagonal, they commute with diagonal matrices. As such, for every diagonal matrix $X \in M_d(\C)$ and $k \in \{1, \dots, d\}$, we have:
$$ L_k X L_k^* - \half \{L_k^* L_k, X \} = 0 $$
and thus $0 = \Lc \Pi (X) = \Lc (X) = -\imath [H, X]$. Hence the Hamiltonian $H$ commutes with all diagonal matricesUsing, ensuring it is diagonal itself. We have thus shown that $\Pi\Lc=\Lc\Pi=0$ implies diagonal Kraus operators and diagonal Hamiltonian. 
Reciprocally, let us assume now that $H$ and every $L_k$ is diagonal. By the first item of this lemma we have that  $\Lc (E_{i,j}) ={\tau}_{i,j} E_{i,j}$. For any $i,j$, we have then that $(\Pi \Lc)(E_{i,j}) =\delta_{i,j} {\tau}_{i,j} E_{i,i}=0$ since $\tau_{i,i}=0$. Hence $\Pi \Lc =0$. Similarly, for any $i,j$, we have $(\Lc \Pi)(E_{i,j}) =\delta_{i,j} {\tau}_{i,j} E_{i,i}=0$, so that $\Lc \Pi =0$.

\bigskip
It remains to prove that $Q=\Pi\Lc^*\Pi$ is the generator of a Markov process on the pointer states. It follows from Eq.~\eqref{eq:missing_formula}, that $Q_{i,j}=(\Lc(E_{i,i}))_{j,j}$ is non-negative for any $i\neq j$ and $\sum_{j=1}^d Q_{i,j}=0$. Thus, the lemma is proved.
\end{proof}

\begin{proof}[Proof of Proposition \ref{proposition:Lindblad_perturbation}]
We proceed in two steps. First, we show that the convergence of $e^{t \Lc_\gamma}$ is a consequence of Corollary~\ref{cor:homo Lindbladian}, by carefully checking all its hypotheses. Second, we will compute the explicit expression of $\Lc_{\infty}$ that will be expressed in terms of the Markov generator $T$.

\bigskip

{\bf Step 0: Checking the hypotheses of Corollary~\ref{cor:homo Lindbladian}.}

Thanks to the QND Assumption \ref{ass:(Nd)} and Lemma \ref{lemma:lindblad_structure} applied with $\Lc=\Lc^{(2)}$, we get that $\Lc^{(2)}$ is diagonalizable with the eigenvalues $\tau_{i,j}$ given by Eq.~\eqref{eq:L_eigenequation2} and associated eigenvector  given by $E_{i,j}=e_i e_j^*$, $1\le i, j \le d$.

\begin{enumerate}[i)]
\item From the expression of $\tau_{i,j}$, the absence of purely imaginary eigenvalues occurs if and only if for every $ 1\le i,j \le d$,
$$\left(\forall k\in \{1,\dotsc,\ell_2\},\ (L_k^{(2)})_{i,i}=(L_k^{(2)})_{j,j}\right)\; \Longrightarrow \; H^{(2)}_{i,i}=H^{(2)}_{j,j}.$$
This implication trivially holds thanks to the Decoherence Condition in Assumption \ref{ass:Dec}.
\item The centering assumption iii) of Theorem~\ref{thm:homogenization} required is a consequence of item 2) in Lemma~\ref{lemma:lindblad_structure} and the QND Assumption \ref{ass:(Nd)}. 

\end{enumerate}

\bigskip

{\bf Step 1: Identification of $\Pc$ of Theorem~\ref{thm:homogenization} and Corollary~\ref{cor:homo Lindbladian} with $\Pi$.}

From the expression in Eq.~\eqref{eq:L_eigenequation2} of $\tau_{i,j}$ we get that $\tau_{i,j}=0$ if and only if $(L_k^{(2)})_{i,i}=(L_k^{(2)})_{j,j}$ and $H^{(2)}_{i,i}=H^{(2)}_{j,j}$. The Decoherence Condition in Assumption \ref{ass:Dec} implies this is equivalent to $i=j$ and it also follows then that $\Ker \Lc^{(2)}$ is the vector space spanned by the eigenvectors $(E_{i,i})_{i=1}^d$, i.e. the set of matrices diagonal in the pointer basis. Moreover the eigenvectors $E_{i,j}$ are orthogonal for the Hilbert-Schmidt scalar product on $M_d (\C)$. It follows, by the step $0$ of the proof of Theorem \ref{thm:homogenization}, that the projector $\Pc$ coincides in the current context to the orthogonal projector $\Pi$ onto ${\rm{Span}}_{\mathbb C} (E_{i,i}\; ; \; 1\le i \le d)$, i.e. for all $X\in M_{d} (\C)$, 
\begin{equation}
\label{eq:defiinitionPi}
\Pi (X)= \sum_{i=1}^d X_{i,i} E_{i,i}.
\end{equation}

\bigskip

{\bf Step 2: The explicit expression of $\Lc_\infty$ matches $T$.}

By Corollary~\ref{cor:homo Lindbladian} and Theorem~\ref{thm:homogenization} we have that 
$$
\Lc_\infty= \Pi\Lc^{(0)}\Pi-\Pi\Lc^{(1)}\left( \Lc^{(2)} \right)^{-1}\Lc^{(1)}\Pi.
$$

Since $\Pi$ is the orthogonal projection onto the vector space of diagonal matrices, it is sufficient to compute $\Lc_\infty (E_{i,i})$. Since the latter must be a diagonal matrix, we are reduced to computing $(\Lc_\infty (E_{i,i}))_{j,j}$ and showing:
\begin{align}
\label{eq:L_equals_T}
(\Lc_\infty (E_{i,i}))_{j,j} = T_{i,j} \ ,
\end{align}
where $T_{i,j}$ is given by Eq.~\eqref{eq:T-rates}.

Since $\Pi^*(\id)=\id$ and $\Lc^*(\id)=0$ for any Lindbladian $\Lc$, we have $\Lc^*_{\infty} (\id)=0$ and consequently for any $i$, $\sum_{j=1}^d (\Lc_\infty (E_{i,i}))_{j,j} =0$. Hence it is sufficient to compute $(\Lc_\infty (E_{i,i}))_{j,j}$ for $i \ne j$. We have:
\begin{align*}
(\Lc_\infty (E_{i,i}))_{j,j}= &\left( \Lc^{(0)} (E_{i,i}) \right)_{j,j}
           -\left( \left(\Lc^{(1)}\left( \Lc^{(2)} \right)^{-1}\Lc^{(1)}\right)  (E_{i,i}) \right)_{j,j} \\
        = & \sum_{k=1}^{\ell_0} | ( L_{k}^{(0)})_{j,i} |^2
           -\left( \left(\Lc^{(1)}\left( \Lc^{(2)} \right)^{-1}\Lc^{(1)} \right) ( E_{i,i} ) \right)_{j,j} 
\end{align*}
where in the second equality we used  Eq.~\eqref{eq:missing_formula} (proved during the proof of item (2) of Lemma \ref{lemma:lindblad_structure} proved later)
\begin{align*}
      \left( \Lc^{(0)}(E_{i,i}) \right)_{j,j}
  & = \sum_{k=1}^\ell \left[ \left| \langle L^{(0)}_k e_i, e_j \rangle \right|^2 - \delta_{i,j} \left\Vert L_k^{(0)} e_j\right \Vert^2 \right] \ .
\end{align*}

Given the structure of $\Lc^{(1)}$, all Kraus operators being diagonal, we have that the diagonal $E_{i,i}$ matrix maps to:
$$ \Lc^{(1)} (E_{i,i}) =-\imath  [H^{(1)} , E_{i,i} ] = -\imath \sum_{\substack{1 \leq j \leq n \\ i \neq j}} \left( (H^{(1)})_{j,i} E_{j,i} - (H^{(1)})_{i,j} E_{i,j} \right) \ .$$
Hence, $\Lc^{(2)}$ being diagonalizable in the natural matrix basis, the pseudo-inverse $\left( \Lc^{(2)} \right)^{-1}$ yields:
$$\left( \left( \Lc^{(2)} \right)^{-1} \Lc^{(1)} \right) (E_{i,i})
 = -\imath \sum_{\substack{1 \leq j \leq n \\ i \neq j}}
     \left( \frac{H^{(1)}_{j,i} }{\tau_{j,i}} E_{j,i}
          - \frac{H^{(1)}_{i,j} }{\tau_{i,j}} E_{i,j}
     \right)
 \ .
$$
Now, before applying $\Pi \Lc^{(1)}$, recall that the Kraus operators associated to $\Lc^{(1)}$ are diagonal. As such, the non-Hamiltonian part of $\Lc^{(1)}$:
$$ X \mapsto \sum_{k=1}^{\ell_1} L_k^{(1)} X (L_k^{(1)})^* - \half \{ (L_k^{(1)})^* L_k^{(1)}, X\}$$
is diagonalizable in the pointer state basis $\left(E_{i,j}\right)_{i,j=1}^d$. Therefore, since 
$$\left( \left( \Lc^{(2)} \right)^{-1} \Lc^{(1)}\right) (E_{i,i}) \in \Span_\C\left( E_{i,j}, E_{j,i}  \ ; \ i \neq j \right) \ ,$$
only the Hamiltonian part $X \mapsto -\imath[H_1, X]$ contributes in:
\begin{align*}
    \ - \left(\Pi \Lc^{(1)} \left( \Lc^{(2)} \right)^{-1} \Lc^{(1)} \right)(E_{i,i})
 = & \ +\imath\  \Pi  \left[ H^{(1)},  \left(\left( \Lc^{(2)} \right)^{-1} \Lc^{(1)} \right) (E_{i,i}) \right] \\
 = & \ 
     \sum_{\substack{1 \leq j \leq n \\ i \neq j}}
     \left( \frac{H^{(1)}_{j,i}}{\tau_{j,i}} \Pi [H^{(1)}, E_{j,i}]
          - \frac{H^{(1)}_{i,j}}{\tau_{i,j}} \Pi [H^{(1)}, E_{i,j}]
     \right) \ .
\end{align*}
Using the fact that:
$$ \Pi [H^{(1)}, E_{i,j}]
 = \left( H^{(1)} \right)_{j,i} \left( E_{j,j} - E_{i,i} \right) \ ,
$$
that $H=H^*$ and $\tau_{i,j} = \overline{ \tau_{j,i} }$, the computation continues as follows:
\begin{align*}
   & \ - \left(\Pi \Lc^{(1)} \left( \Lc^{(2)} \right)^{-1} \Lc^{(1)} \right) (E_{i,i})\\
 = & \
     \sum_{\substack{1 \leq j \leq n \\ i \neq j}}
     \left( \frac{H^{(1)}_{j,i}}{\tau_{j,i}}  H^{(1)}_{i,j} \left( E_{i,i} - E_{j,j} \right) 
          - \frac{H^{(1)}_{i,j}}{\tau_{i,j}} H^{(1)}_{j,i} \left( E_{j,j} - E_{i,i} \right)
     \right)
     \\
     =& \sum_{\substack{1 \leq j \leq n \\ i \neq j}} \left( \frac{1}{\tau_{j,i}} + \frac{1}{\tau_{i,j}} \right) \Big\vert H^{(1)}_{i,j}\Big\vert^2 (E_{i,i}-E_{j,j}) \\
     =&  \sum_{\substack{1 \leq j \leq n \\ i \neq j}} \frac{2\Re(\tau_{i,j})}{|\tau_{i,j}|^2}\Big\vert H^{(1)}_{i,j} \Big\vert^2 (E_{i,i}-E_{j,j}) \ .
\end{align*}
Putting everything together, we find the result:
\begin{align*}
(\Lc_\infty (E_{i,i}))_{j,j}= & \sum_{k=1}^{\ell_0} | ( L_0^{(k)})_{j,i} |^2
           -\frac{2\Re(\tau_{i,j})}{|\tau_{i,j}|^2}
     \left| H^{(1)}_{i,j} \right|^2 \\
        = & \sum_{k=1}^{\ell_0} | ( L_0^{(k)})_{j,i} |^2
           +\frac{\left|  H^{(1)}_{i,j} \right|^2}{|\tau_{i,j}|^2}
            \sum_{k=1}^{\ell_2} \left| (L_k^{(2)})_{i,i} - (L_k^{(2)})_{j,j} \right|^2 \ .
\end{align*}
Finally we recognize the transition rate $T_{i,j}$ defined in Eq.~\eqref{eq:T-rates}.

\bigskip

{\bf Step 3: Mean convergence of $\rho^\gamma$.}

Let us now prove Eq.~\eqref{eq:mean-rho}. By exponentiating the diagonal operator $\Lc_\infty$, Eq.~\eqref{eq:L_equals_T} implies that 
\begin{align*}
    \left( e^{t \Lc_\infty} (E_{i,i}) \right)_{j,j} 
& = \left( e^{t T} \right)_{i,j} 
  = \P_{e_i} \left[ \xb_t = e_j \right] \ .
\end{align*}

Hence by linearity we have
\begin{align*}
    \left(\Pi e^{t\Lc_\infty} \Pi \right) (\varrho)
= & \ \sum_{i,j=1}^d {\varrho}_{i,i} (e^{t \Lc_\infty} (E_{i,i}))_{j,j} E_{j,j}\\
= & \ \sum_{i,j=1}^d \varrho_{i,i} {\mathbb P}_{e_i} \left[ \xb_t = e_j \right] E_{j,j}\\
= & \ \E_{\mu_{\varrho}} \left[\xb_t \xb_t^* \right].
\end{align*}

\end{proof}

\section{Proof of Main Theorem \ref{thm:main}}
\label{section:proof}

In this section, we tackle the proof of the main theorem. Let us give a quick sketch of proof, which is done in five steps, each corresponding to its own subsection.

\begin{itemize}

\item {\bf Step 1 in Subsection \ref{subsection:decomposition_of_traj}:} There, we lay the groundwork thanks to a convenient trajectorial decomposition. Indeed, thanks to the QND Assumption \ref{ass:(Nd)}, many super-operators act diagonally in the pointer states basis $(E_{i,j})_{i,j=1}^d$, and it will be convenient to write $\rho^\gamma$ as a perturbation of a diagonal flow on matrices $\left( T^\gamma_{0, t} \ ; \ t \geq 0 \right)$. This is the content of Proposition \ref{proposition:traj_decomposition}.

\item {\bf Step 2 in Subsection \ref{subsection:tightness}:} Thanks to a tightness criterion which is recalled in the text, Proposition \ref{proposition:tightness} handles tightness of $\rho^\gamma$ in the Meyer-Zheng topology. A key technical point, Lemma \ref{lemma:lyapounov} is used but proven later.

In passing, we confirm the physical intuition of decoherence by showing that any limiting subsequence must have zero off-diagonal coefficients.

\item {\bf Step 3 in Subsection \ref{subsection:gronwall}:} This subsection contains the proof of Lemma \ref{lemma:lyapounov}.

\item {\bf Step 4 in Subsection \ref{subsection:time_control}:} There, we prove that, on average, the time spent by $\rho^\gamma$ away from the pointer states $\left( E_{i,i} \right)_{i=1}^{d}$ vanishes as $\gamma \rightarrow \infty$.

\item {\bf Step 5 in Subsection \ref{subsection:fd_distributions}:} Finally, we identify the limiting process $\xb \xb^*$ thanks to its finite-dimensional distributions. The last step crucially relies on the convergence of the mean established in Proposition \ref{proposition:Lindblad_perturbation}.

\end{itemize}

In this whole section, the projection $\Pi$ onto diagonal matrices in the pointer basis $(e_i)_{i=1}^d$, is actually an orthogonal projection: $\Pi=\Pi^*$.

\subsection{Decomposition of the trajectories}
\label{subsection:decomposition_of_traj}
We define a stochastic flow $(T^\gamma_{s,t})_{t\geq s}$ taking value in $\End( M_d(\C) )$ that we will use to decompose the evolution of $\rho^\gamma$. This process will induce a time exponential decay of the off diagonal elements of the state at a rate of order $\gamma^2$. Furthermore, it will be particularly tractable as $T^\gamma$ will act diagonally in the natural basis $\left( E_{i,j} \right)_{i,j=1}^d$, and will thus fully leverage the QND Assumption.

The definition of $T^\gamma$ requires a few ingredients. First, define the non-Hamiltonian i.e. dissipative part of $\Lc^{(1)}$ as:
$$ \Lc_{\textrm{diss}}^{(1)}:X\mapsto \sum_{k=1}^{\ell_1} L_k^{(1)} X (L_k^{(1)})^* -\frac12\left\{(L_k^{(1)})^*L_k^{(1)},X\right\} \ .$$
Now recall from the QND Assumption \ref{ass:(Nd)} that Kraus operators associated to $\Lc^{(i)}$, $i=1,2$ and $H^{(2)}$ are all diagonal. As such, a fruitful idea is to use $\Lc_{\textrm{diss}}^{(1)}$ and $\Lc^{(2)}$ in order to form the stochastic flow $T^\gamma$. 

Second, we shall require a certain modification in the volatilities in Eq.~\eqref{eq:def_sigma}. For $\rho$ and $\xi$ two matrices in $M_d(\C)$, define $\sigma^{(\alpha)} (\rho, \xi) :=\left(\sigma^{(\alpha)}_k (\rho, \xi) \right)_{k=1}^{\ell_\alpha}  \in (M_d (\C))^{\ell_\alpha} $ by:
\begin{align}
\label{eq:def_sigma_two}
\sigma^{(\alpha)}_k (\rho, \xi) = & \ {\sqrt{\eta_\alpha (k)}} \left(L^{(\alpha)}_k \xi+\xi L^{(\alpha)\, *}_k-\tr \Big[(L^{(\alpha)\, *}_k+L^{(\alpha)}_k)\rho \Big]\xi \right) \ ,
\end{align}
so that the above quantity is linear in $\xi$. We write $\sigma_\alpha(\rho, \rho) = \sigma_\alpha(\rho)$, which coincides with the notation throughout the paper.

For any $\gamma>0$ and $s\geq 0$, let $(T^\gamma_{s,t})_{t\geq s}$ be a process taking value in $\textrm{End}( M_d(\C) )$ and solution to (recall Section \ref{subsec:not} for the notation below) 
\begin{align}
\left\{
\begin{array}{ccc}
T^\gamma_{s,s}  & = & \id_{M_d(\C)} \\
dT^\gamma_{s,t} & = & \Big( (\gamma\Lc_{\textrm{diss}}^{(1)}+\gamma^2 \Lc^{(2)} ) \, dt + \gamma^{\frac12}\sigma^{(1)}(\rho^\gamma_t, \cdot)\cdot dW^{1}_t \\
                &   & \quad +\gamma\sigma^{(2)} (\rho^\gamma_t, \cdot)\cdot dW^{2}_t \Big) \circ T^\gamma_{s,t} ,
\end{array}
\right.
\label{eq:def_T}
\end{align}
where $\rho^\gamma$ is the solution to the SDE \eqref{eq:SDE} and $\circ$ is the composition in $\textrm{End}( M_d(\C) )$. Notice that $\gamma\Lc_{\textrm{diss}}^{(1)}+\gamma^2 \Lc^{(2)}$ are understood as elements of $\textrm{End}( M_d(\C) )$ and the same goes for $\gamma^{\frac12}\sigma^{(1)} (\rho^\gamma_t, \cdot)\cdot dW^{(1)}_t + \gamma\sigma^{(2)} (\rho^\gamma_t, \cdot )\cdot dW^{(2)}_t$ which are (infinitesimal) linear operator of  $\textrm{End}( M_d(\C) )$. Here the $\cdot$ between parentheses denotes the free variable in $M_d (\C)$.

We can now state the main result of this subsection:
\begin{proposition}
\label{proposition:traj_decomposition}
Under our working assumptions, the following hold:
	
\begin{enumerate}[1)]

	\item For any $\gamma>0$, the stochastic differential equation \eqref{eq:def_T} has a unique strong solution. 

	\item The stochastic flow $T^\gamma_{0, \cdot} = \left( T^\gamma_{0,t} \ ; \ t \geq 0 \right)$ acts diagonally in the matrix basis $\left( E_{i,j} \right)_{i,j=1}^d$ i.e.
\begin{align}
\label{eq:T_is_diagonal}
    T^\gamma_{0,t}\left( E_{i,j} \right)
= & \ \left[ T^\gamma_{0,t} \right]_{i,j} E_{i,j} \ ,
\end{align}
for $(i,j)\in\{1,\dotsc,d\}^2$. Moreover $\left[ T^\gamma_{0,t} \right]_{i,j}$ is a semimartingale of the form
\begin{align}
	\label{eq:expression_T}
	    \left[ T^\gamma_{0,t} \right]_{i,j}
	    =&  \ e^{ -\frac t2 D^\gamma_{i,j} + \imath \theta_{i,j}^\gamma(t) }
 \ \Ec\left( \sum_{\alpha=1}^2 \int_0^t \gamma^{{\alpha}/{2}} \, a_{i,j}^{\alpha,\gamma} (s) \cdot dW_s^\alpha \right)  \ ,
%
\end{align}
    where $ a_{i,j}^{\alpha,\gamma}$ are real semimartingales, taking value in compact subsets, independent of $\gamma$, of respectively $\R^{\ell_1}$ and $\R^{\ell_2}$; $\Ec(M)_t := \exp(-\frac12\langle M \rangle_t+M_t-M_0)$ is the Dol\'eans--Dade exponential for any real continuous martingale $M$; $\theta^\gamma_{i,j}$ is a real continuous semimartingale. Finally, for any $(i,j)\in\{1,\dotsc,d\}^2$, the constant $D_{i,j}^\gamma$ is given by
	\begin{align}
	\label{eq:exprDij}
  D^\gamma_{i,j} & = \ \sum_{\alpha=1}^2 \gamma^{\alpha} \, D_{i,j}^{\alpha, \gamma} 
    \end{align}
   where:
   \begin{align*}
       D_{i,j}^{\alpha, \gamma}
   & = \  \sum_{k=1}^{\ell_\alpha} \Re \left[ (L_k^{(\alpha)})_{i,i} - (L_k^{(\alpha)})_{j,j} \right]^2 \\
   & \quad \quad 
        + \sum_{k=1}^{\ell_\alpha} (1 - \eta_\alpha(k)) \Im \left[ (L^{(\alpha)}_k)_{i,i} - (L^{(\alpha)}_k)_{j,j} \right]^2 .
\end{align*}

    \item For any $s\geq 0$, $T^\gamma_{0,s}$ is invertible and $T^\gamma_{s,t}=T^\gamma_{0,t}\circ (T^\gamma_{0,s})^{-1}$. 
    
    \item The stochastic flow $T^\gamma_{0, \cdot}$ allows to express the solution $\rho^\gamma$ to the SDE \eqref{eq:SDE} as
\begin{align}
\label{eq:rho_T_perturbation}
    \rho_t^\gamma
= & \ T_{0,t}^\gamma \left( \varrho + \int_0^t  \left[(T_{0,s}^\gamma)^{-1}\circ(\Lc^{(0)}-\imath \gamma \operatorname{ad}_{H^{(1)}}) \right] (\rho_s^\gamma)\,  ds \right. \\
  & \left.\qquad \qquad + \int_0^t(T_{0,s}^\gamma)^{-1}\Big( \sigma^{(0)} (\rho^\gamma_s) \cdot dW_s^0\Big)\;   \right)
    \nonumber
\end{align}
where $\operatorname{ad}_H:X\in M_d(\C)\mapsto [H,X]$ is the adjoint action associated to a matrix $H \in M_d(\C)$.
\end{enumerate}
\end{proposition}
\begin{rmk}
In the spirit of Duhamel's principle, the full dynamic $\rho^\gamma$ should be a perturbation of $T^\gamma$, which is Eq. \eqref{eq:rho_T_perturbation}. The perturbation term between parenthenses is exactly the non-QND part of the dynamic. 
\end{rmk}
\begin{proof}
    We prove each claim sequentially:
    
    \medskip
    
	1) The stochastic differential equation \eqref{eq:def_T} admits a unique strong solution, because given a fixed $\rho^\gamma$, the coefficients are globally Lipschitz. 
	
	\medskip
	
	2) 	Given the QND Assumption \ref{ass:(Nd)} on $\Lc^{(1)}$ and $\Lc^{(2)}$, and Lemma \ref{lemma:lindblad_structure}, for any $\rho\in\Sc$, $\Lc_{\textrm{diss}}^{(1)}$, $\Lc^{(2)}$ are all diagonal in the basis $\left( E_{i,j} \right)_{i,j=1}^d$ of $M_d(\C)$. The same holds for all the components of $\sigma^{(1)} (\rho , \cdot)$ and $\sigma^{(2)} (\rho , \cdot)$. Therefore, Eq.~\eqref{eq:def_T} can be written as a system of $d^2$ independent equations, one for each $E_{i,j}$. As such, Eq.~\eqref{eq:T_is_diagonal} clearly holds.
	
	We now fix arbitrarily such $i$ and $j$, and then proceed with a much closer inspection. Recalling from the QND Assumption \ref{ass:(Nd)} and Lemma \ref{lemma:lindblad_structure} that $E_{i,j}$ is an eigenvector for the Lindbladians with explicit eigenvalues:
$$ \Lc^{(1)}_{\rm{diss}} \left( E_{i,j} \right) = \tau_{i,j}^{(1)} E_{i,j} \ , \quad
   \Lc^{(2)}\left( E_{i,j} \right) = \tau_{i,j}^{(2)} E_{i,j} \ . 
$$	
Moreover, from Eq.~\eqref{eq:def_sigma_two}, using the fact that Kraus operators are diagonal, we have for $\alpha=1,2$ and $k\in \{1, \ldots, \ell_\alpha\}$:
\begin{align*}
     \sigma_k^{(\alpha)} (\rho_t^\gamma, E_{i,j})
= & \ \left( (L^{(\alpha)}_k)_{i,i} + (L^{(\alpha)\, *}_k)_{j,j}-\tr \Big[(L^{(\alpha)\, *}_k+L^{(\alpha)}_k)\rho_t^\gamma \Big] \right) {\sqrt{\eta_\alpha (k)}} E_{i,j} \ .
\end{align*}
As such, we obtain martingale increments for $\alpha=1,2$
\begin{align*}
   & \sigma^{(\alpha)} (\rho_t^\gamma, E_{i,j}) \cdot dW_t^\alpha \\
 = & \ E_{i,j} \; \sum_{k=1}^{\ell_\alpha} \left( (L^{(\alpha)}_k)_{i,i} + (L^{(\alpha)\, *}_k)_{j,j}-\tr[(L^{(\alpha)\, *}_k+L^{(\alpha)}_k)\rho_t^\gamma] \right) {\sqrt{\eta_\alpha (k)}} dW_t^{\alpha,k}  \\
:= & \ E_{i,j} \ z_{i,j}^{\alpha, \gamma} (t) \cdot d W_t^\alpha \ ,
\end{align*}
where for $\alpha=1,2$, $z^{\alpha, \gamma}_{i,j}$ is a $\C^{\ell_\alpha}$ valued predictable process. Observe that since $\rho^\gamma_t$ belongs to the compact set $\mathcal S$ for any $t\ge 0$, each process $z^{\alpha, \gamma}_{i,j}$ lives in a compact subset of $\C$.

\medskip

In the end, the system of Eq.~\eqref{eq:def_T} becomes:
$$ d\left[ T_{0,t}^\gamma \right]_{i,j}
  = \left[ T_{0,t}^\gamma \right]_{i,j}
    \ \sum_{\alpha=1}^2 \left\{  \gamma^{\alpha} \tau_{i,j}^{(\alpha)} \ dt \; + \;  \gamma^{\alpha/2}\,  z_{i,j}^{\alpha, \gamma} (t) \cdot dW^\alpha_t \right\}.$$    
Using It\^o calculus:
 \begin{align*}
&	\left[ T_{0,t}^\gamma \right]_{i,j} =\\
&\exp{\left( \sum_{\alpha=1}^2 \left\{  \gamma^{\alpha} \tau_{i,j}^{(\alpha)} \ t \; + \;  \int_0^t \gamma^{\alpha/2}\,  z_{i,j}^{\alpha, \gamma} (s) \cdot dW^\alpha_s - \half   \sum_{k=1}^{\ell_\alpha}  \int_0^t \gamma^{\alpha}\,  (z_{i,j}^{\alpha, \gamma} (s))_k^2 \, ds
	           \right\} \right)} \ .
\end{align*}   
We isolate the imaginary part in the exponential as the real semimartingale $\theta_{i,j}^\gamma$, and define the real semimartingales 
$$a_{i,j}^{\alpha, \gamma} (t) :=\Re z_{i,j}^{\alpha, \gamma} (t), \quad b_{i,j}^{\alpha, \gamma} (t) :=\Im z_{i,j}^{\alpha, \gamma} (t),$$
so that:
\begin{align*}
	\left[ T_{0,t}^\gamma \right]_{i,j}
& = e^{\imath \theta^\gamma_{i,j}(t) + t \Re \big( \gamma \tau_{i,j}^{(1)} + \gamma^2 \tau_{i,j}^{(2)} \big)} \; \Ec\left(  \sum_{\alpha=1}^2 \int_0^t \gamma^{\alpha/2}  a^{\alpha,\gamma}_{i,j}(s) \cdot dW^\alpha_s\right)\\
    & \quad \quad \times\, \exp\left( 
	           \half   \sum_{\alpha=1}^2 \sum_{k=1}^{\ell_\alpha}  \int_0^t \gamma^{\alpha}\,  (b_{i,j}^{\alpha, \gamma} (s))_k^2 \, ds
	            \right) \ .
\end{align*}
At this point, we see that we are very close to the result. All that remains is to handle the last exponential term and force the appearance of the term $D^\gamma_{i,j}$. To that endeavor, it is crucial to notice that because $\rho_s^\gamma$  is a density matrix and therefore Hermitian, $\tr \left[(L^{(\alpha)\, *}_k+L^{(\alpha)}_k)\rho_s^\gamma\right]$ is real for $\alpha=1,2$ and thus
\begin{align*}
    \left[  (b_{i,j}^{\alpha, \gamma}(s))_k \right]^2
= & \ \eta_\alpha(k) \Im^2\left[ (L^{(\alpha)}_k)_{i,i} + (L^{(\alpha)\, *}_k)_{j,j}-\tr[(L^{(\alpha)\, *}_k+L^{(\alpha)}_k)\rho_s^\gamma]) \right]\\
= & \ \eta_\alpha(k) \Im^2\left[ (L^{(\alpha)}_k)_{i,i} - (L^{(\alpha)}_k)_{j,j} \right] \ .
\end{align*}
In the end:
\begin{align*}
	\left[ T_{0,t}^\gamma \right]_{i,j}
& = \exp\left( \imath \theta_{i,j}^\gamma (t) - \frac{t}{2}\,  D^\gamma _{i,j} \right)
    \  \Ec\left(  \sum_{\alpha=1}^2 \int_0^t \gamma^{\alpha/2}\,  a_{i,j}^{\alpha,\gamma} (s) \cdot dW^\alpha_s\right) \ .
\end{align*}
where
\begin{align*}
D_{i,j}^\gamma  & := \ - 2 \Re \left( \gamma \tau_{i,j}^{(1)} + \gamma^2 \tau_{i,j}^{(2)} \right) 
             \ - \gamma   \sum_{k=1}^{\ell_1} \eta_1(k) \Im^2\left[ (L^{(1)}_k)_{i,i} - (L^{(1)}_k)_{j,j} \right] \\
	    &    \ \quad \quad - \gamma^2 \sum_{k=1}^{\ell_2} \eta_2(k) \Im^2\left[ (L^{(2)}_k)_{i,i} - (L^{(2)}_k)_{j,j} \right]  \ .
\end{align*}
Thanks to the expressions of eigenvalues in Eq.~\eqref{eq:L_eigenequation}, this can be made more explicit as:
\begin{align*}
  & D_{i,j}^\gamma \\
  & = \ \gamma     \sum_{k=1}^{\ell_1} \left( \left| (L_k^{(1)})_{i,i} - (L_k^{(1)})_{j,j} \right|^2 
                                                      - \eta_1(k) \left[ \Im(L^{(1)}_k)_{i,i} - \Im (L^{(1)}_k)_{j,j} \right]^2 \right)\\
  & \quad   \ + \gamma^2 \sum_{k=1}^{\ell_2} \left( \left| (L_k^{(2)})_{i,i} - (L_k^{(2)})_{j,j} \right|^2 
                                                      - \eta_2(k) \left[ \Im(L^{(2)}_k)_{i,i} - \Im (L^{(2)}_k)_{j,j} \right]^2 \right)\\
        & = \ \gamma     \sum_{k=1}^{\ell_1} \left( \Re^2\left[ (L_k^{(1)})_{i,i} - (L_k^{(1)})_{j,j} \right]
                                                    + (1 - \eta_1(k)) \Im^2 \left[ (L^{(1)}_k)_{i,i} - (L^{(1)}_k)_{j,j} \right] \right)\\
  & \quad   \ + \gamma^2 \sum_{k=1}^{\ell_2} \left( \Re^2\left[ (L_k^{(2)})_{i,i} - (L_k^{(2)})_{j,j} \right]
                                                    + (1 - \eta_2(k)) \Im^2\left[ (L^{(2)}_k)_{i,i} - (L^{(2)}_k)_{j,j} \right] \right).
\end{align*}
This is the required result and we have finished proving claim 2).

	\medskip
	
	3) By virtue of claim 2), the diagonal coefficient $\left[ T^\gamma_{0,\cdot} \right]_{i,j}$ never vanishes and therefore $\left( T^\gamma_{0,s} \right)^{-1}$ exists for all $s \geq 0$.  By uniqueness of the solution to Eq.~\eqref{eq:def_T} we obtain the semigroup property $T_{s,t}^\gamma=T_{0,t}^\gamma\circ(T_{0,s}^\gamma)^{-1}$ for any $0\le s \le t$. 
	
	\medskip
	
	4) For Eq.~\eqref{eq:rho_T_perturbation}, we proceed by differentiating the RHS. For the purpose of the computation, we denote it by $\xi^\gamma$:
\begin{equation*}
\begin{split}
\xi^\gamma_t := T_{0,t}^\gamma &\left( \varrho + \int_0^t \Big[(T_{0,s}^\gamma)^{-1}\circ(\Lc^{(0)}-\imath \gamma \operatorname{ad}_{H^{(1)}}) \Big] (\rho_s^\gamma)\,  ds \right.\\
&\left. \quad
             + \int_0^t(T_{0,s}^\gamma)^{-1}\Big(\sigma^{(0)} (\rho^\gamma_s, \rho^\gamma_s)  \cdot dW_s^0\Big)\;  \right) \ . 
             \end{split}
\end{equation*}
	Since in the definition of $T^\gamma$ only the Brownian motions $W^1$ and $W^2$ are involved,  $T^\gamma_{0,\cdot}$ and $W^0$ have zero quadratic variation, there is no cross-term upon differentiating the semimartingales:
\begin{align*}
    d\xi^\gamma_t
= &	\ \left[dT^\gamma_{0,t} \circ \left( T^\gamma_{0,t} \right)^{-1}\right] \left( \xi^\gamma_t \right)
  + T^\gamma_{0,t} \left(  d \Big[ \varrho + \int_0^t \left[ (T_{0,s}^\gamma)^{-1}\circ(\Lc^{(0)} -\imath \gamma\operatorname{ad}_{H^{(1)}}) \right] (\rho_s^\gamma) ds \right.\\
  & \left. \qquad \qquad\qquad\qquad \qquad\qquad\qquad+ \int_0^t (T_{0,s}^\gamma)^{-1} \big(\sigma^{(0)} (\rho^\gamma_s, \rho_s^\gamma) \cdot dW_s^0 \big)\Big] \; \right)\\
= & \ \left[dT^\gamma_{0,t} \circ \left( T^\gamma_{0,t} \right)^{-1}\right] \left( \xi^\gamma_t \right)
  + T^\gamma_{0,t} \Big( \Big[(T_{0,t}^\gamma)^{-1} \circ (\Lc^{(0)}-\imath \gamma \operatorname{ad}_{H^{(1)}}) \Big] (\rho_t^\gamma) dt \\
  & \qquad \qquad \qquad \qquad\qquad \qquad\qquad+ (T_{0,t}^\gamma)^{-1} \big(\sigma^{(0)} (\rho^\gamma_t, \rho_t^\gamma) \cdot dW_t^0 \big) \Big)\\
= & \ \left[dT^\gamma_{0,t} \circ \left( T^\gamma_{0,t} \right)^{-1}\right] \left( \xi^\gamma_t \right)
  + (\Lc^{(0)}-i \gamma\operatorname{ad}_{H^{(1)}})(\rho_t^\gamma) dt 
  + \sigma^{(0)} (\rho^\gamma_t, \rho_t^\gamma)\cdot dW_t^0 \ .
\end{align*}

Using the convention from Eq.~\eqref{eq:def_sigma_two} and upon plugging in the expression $dT^\gamma_{0,t} \circ \left( T^\gamma_{0,t} \right)^{-1}$ from Eq.~\eqref{eq:def_T}, we find:
\begin{align*}
	d\xi^\gamma_t
= & \ (\gamma\Lc^{(1)}_{\rm diss}+\gamma^2 \Lc^{(2)} )(\xi^\gamma_t) dt + (\Lc^{(0)} -i\gamma\operatorname{ad}_{H^{(1)}})(\rho_t^\gamma)dt\\
  & \quad + \sigma^{(0)}(\rho^\gamma_t, \rho_t^\gamma) \cdot dW_t^0 
          + \gamma^{\frac12}\sigma^{(1)} (\rho^\gamma_t, \xi^\gamma_t) \cdot dW^{1}_t 
          + \gamma\sigma^{(2)} (\rho^\gamma_t, \xi^\gamma_t) \cdot dW^{2}_t \ .
\end{align*}
Notice this is not the same equation as the SDE \eqref{eq:SDE}. Nevertheless, the coefficients are clearly Lipschitz in $\xi^\gamma$, as we consider $\rho^\gamma$ as an autonomous term. Therefore, the above equation enjoys the uniqueness property. Obviously $\rho^\gamma$ is a solution and therefore $\rho^\gamma = \xi^\gamma$. We have thus proven Eq.~\eqref{eq:rho_T_perturbation}.
\end{proof}

\subsection{Tightness in Meyer-Zheng topology}
\label{subsection:tightness}

Recall that the process $\rho^\gamma = \left( \rho^\gamma_t \ , \ t \geq 0 \right)$ is a process with continuous in time trajectories and that in order to prove the main Theorem \ref{thm:main}, we aim to show the convergence towards a pure jump Markov process, thus with discontinuous in time trajectories. As already argued in the introduction, this makes the usual Skorokhod topology on c\`adl\`ag paths irrelevant and it motivates the use of the Meyer-Zheng topology defined in Definition \ref{def:MZtopology}. 

Consider the subspace $\D(\R_+ ; E) \subset {\mathbb L}^0 (\R_+ ; E)$ of $E$-valued c\`adl\`ag paths, where both $\rho^\gamma$ and $\xb$ take their values. More specifically $\rho^\gamma \in \D(\R_+ ;  \Sc)$ and $\xb \in \D \left(\R_+ ;  (e_i)_{i=1}^d \right)$. Since $({\mathbb L}^0, {\rm d})$ is a Polish space, by Prokhorov's theorem, the relative compactness of a sequence of probability measures on $({\mathbb L}^0, {\rm d})$ is reduced to the tightness property of this sequence. However the space $\left( \D(\R_+ ; E), {\rm d} \right)$ is not Polish (see Appendix of \cite{MZ} for a proof) so that the relative compactness in this space requires a specific analysis. The latter has been performed in \cite{MZ}.

\medskip

The main result of this subsection takes us closer to the weak convergence of $\rho^\gamma$ to $\xb$ by stating that:

\begin{proposition}
\label{proposition:tightness}
The family of the laws of $\left( \rho^\gamma \ ; \ \gamma > 1 \right)$ is sequentially compact in the space of probability measures on $({\mathbb L}^0, {\rm d})$. Moreover any limiting subsequence is the law of a process valued in diagonal matrices in the pointer basis and with c\`adl\`ag trajectories.
\end{proposition}
\begin{proof}[Strategy of proof]
The proof is subdivided into two steps:
\begin{itemize}
\item Off-diagonal terms $\Pi_{\perp} (\rho^\gamma)$ converge to zero in probability for the Meyer-Zheng topology, which is the content of the upcoming Lemma \ref{lemma:tightness_offdiag}.
\item The diagonal terms $\left( \Pi (\rho^\gamma) \right)_{\gamma>0}$ form a sequentially compact family using a criterion by Meyer-Zheng, which is the content of the upcoming Lemma \ref{lemma:tightness_diag}. 
\end{itemize}
By putting these together, the pair $\left(\Pi_\perp (\rho^\gamma), \Pi (\rho^\gamma)\right)$ forms a tight family indexed by $\gamma>1$, simply because the product of two compact spaces is compact. Therefore the law of $\rho^\gamma = \Pi_\perp (\rho^\gamma) + \Pi (\rho^\gamma)$ is tight. 

Now considering possible accumulation points of the law of $\rho^\gamma$ for $\gamma \rightarrow \infty$, since $\Pi_\perp\rho^\gamma$ accumulates necessarily to zero, by Slutsky's Lemma, any limit of a convergent subsequence will be the law of a process valued in $M_d(\C)$, identically vanishing outside diagonal matrices.
\end{proof}

\begin{lemma}
\label{lemma:tightness_offdiag}
$$ \forall \varepsilon >0, \ 
   \lim_{\gamma \rightarrow \infty}
   \P\left( {\rm d} \left( \Pi_\perp (\rho^\gamma), 0 \right) \ge \varepsilon \right) = 0 \ .
$$
\end{lemma}
\begin{proof}
Recalling that 
$$ {\rm d} \left( \Pi_\perp (\rho^\gamma), 0 \right)
   = 
   \int_0^\infty \Big\{ 1 \wedge \left\| \Pi_\perp (\rho^\gamma)  \right\|\Big\} \ e^{-t} dt\,.
$$
We only need to prove, by Markov's inequality:
$$
   \lim_{\gamma \rightarrow \infty} \int_0^\infty \E \left\| \Pi_\perp (\rho^\gamma) \right\| e^{-t} \ dt = 0 \ .
$$
In turn, this is implied by the fact that there exist $C>0$ and $g>0$ such that for any $t\in \R_+$ and $\gamma>1$,
\begin{align}
\label{eq:offdiag_L1_sqrtgamma}
\E(\|\Pi_\perp (\rho_t^\gamma) \|)\leq & \ C\left(e^{-g\gamma^2 t} +\frac1{\gamma}\right).
\end{align}
which we shall prove under the identifiability Assumption \ref{ass:(Id)} and $\Pi\Lc^{(1)}\Pi=\Pi\Lc^{(2)}=\Lc^{(2)}\Pi=0$.

The proof of Eq.~\eqref{eq:offdiag_L1_sqrtgamma} is the core of the lemma, and will be loosely referred to as decoherence. It is decomposed into three steps.

\medskip

{\bf Step 1: Controlling off-diagonal terms in the flow $T_{0,t}^\gamma$:}
Let's prove that under our working assumptions there exist $C>0$ and $g>0$ such that for any $t\geq s>0$ and $\gamma>0$,
\begin{align}
\label{eq:stoch_off_diag_T_decay}
\E(\|T_{s,t}^\gamma \circ \Pi_\perp\|) & = \E(\|\Pi_\perp \circ T_{s,t}^\gamma\|)\leq Ce^{-g\gamma^2 (t-s)}.
\end{align}

Eq.~\eqref{eq:T_is_diagonal} and item 3) of Proposition \ref{proposition:traj_decomposition} imply that $T_{s,t}^\gamma \in {\rm{End}} (M_d (\C))$ is diagonal in the basis $(E_{i,j})_{i,j=1}^d$ of $M_d (\C)$. The expression of $\Pi$ given in Eq.~\eqref{eq:defiinitionPi} then implies first that $T_{s,t}^\gamma$ commute with $\Pi$, and thus with $\Pi_\perp = \id -\Pi$,  and secondly that the  $1$-norm of $T_{s,t}^\gamma \circ \Pi_{\perp}$ w.r.t. this basis  is given by 
$$ 
\sum_{\substack{1 \leq i, j \leq d \\ i \neq j}} \left\vert  \left[T_{s,t}^\gamma \right]_{i,j}\right\vert \ .
$$
Hence, by equivalence of the norms in finite dimension, it is sufficient to bound the expectation $\E \left(\left\vert  \left[T_{s,t}^\gamma \right]_{i,j}\right\vert \right)$ for an arbitrary pair of distinct $(i,j)\in\{1,\dots,d\}^2$. From Eq.~\eqref{eq:expression_T} and item 3) of Proposition \ref{proposition:traj_decomposition}, 
	\begin{align*}
	\left\vert  \left[T_{s,t}^\gamma \right]_{i,j}\right\vert = \ e^{ -\frac {(t-s)}{2} D^\gamma_{i,j}  }
 \ \Ec\left( \sum_{\alpha=1}^2 \int_s^t \gamma^{{\alpha}/{2}} \, a_{i,j}^{\alpha,\gamma} (u) \cdot dW_u^\alpha \right) 	\end{align*}
	 Then the identifiability Assumption \ref{ass:(Id)} and Eq.~\eqref{eq:exprDij} implies $\min_{i\neq j}D^\gamma_{i,j} \ge c \gamma^2 $ for some $c>0$. Hence, recalling that Dol\'eans--Dade martingales $\Ec$ have expectation $1$, Eq.~\eqref{eq:stoch_off_diag_T_decay} is proven.

\medskip

{\bf Step 2: Bootstrapping decoherence of $T^\gamma$ to $\rho^\gamma$, a beginning.}

	From Eq.~\eqref{eq:rho_T_perturbation} and the triangular inequality, we obtain three terms:
	\begin{align}
	\label{eq:bound_off_diag}
	       \ \E(\|\Pi_\perp \left(\rho_t^\gamma \right)\|) 
	\leq & \ \quad \E\left(\left\| [\Pi_\perp \circ  T_{0,t}^\gamma ] (\varrho)\right\|\right)\\
	     & \ \quad + \int_0^t \E\left( \left\| [\Pi_\perp \circ T_{s,t}^\gamma\circ(\Lc^{(0)}-\imath \gamma \operatorname{ad}_{H^{(1)}}) ]  (\rho_s^\gamma)\right\|\right)ds \nonumber \\
	     &   \quad + \E\left(\left\|\int_0^t [\Pi_\perp \circ T_{s,t}^\gamma] \left( \sigma^{(0)} \, (\rho_s^\gamma) \cdot dW_s^0 \right) \right\|\right) \ . \nonumber
	\end{align}
	The main goal of this Lemma which is Eq.~\eqref{eq:offdiag_L1_sqrtgamma}, will be obtained by controlling the three previous terms. We begin by controlling the two first terms thanks to Step 1, and more specifically Eq.~\eqref{eq:stoch_off_diag_T_decay}. Since $\rho_t^\gamma\in\mathcal S$ almost surely for any $t\in \R_+$ and $\gamma>0$, 
	$$ \|\varrho\|\leq\sup_{t\in\R_+, \gamma>0} \|\rho_t^\gamma\|\leq 1 $$
	almost surely. Then, using Eq.~\eqref{eq:stoch_off_diag_T_decay}, there exists $C>0$ and $g>0$ such that for any $\gamma>0$ and $t\in \R_+$, the first term of the sum in the right hand side of inequality \eqref{eq:bound_off_diag} is upper bounded as
	$$\E\left(\left\| [\Pi_\perp \circ T_{0,t}^\gamma] (\varrho)\right\|\right)\leq C e^{-g\gamma^2 t}.$$
	Since both $\Lc^{(0)}$ and $\operatorname{ad}_{H^{(1)}}$ are bounded, using Eq.~\eqref{eq:stoch_off_diag_T_decay} and integrating the exponential shows that there exists $C>0$ such that for any $t\in\R_+$ and $\gamma>1$, the second term in the sum on the right hand side of inequality \eqref{eq:bound_off_diag} is upper bounded as
	$$\int_0^t\E\left(\left\| [ \Pi_\perp  \circ T_{s,t}^\gamma\circ(\Lc^{(0)}-\imath \gamma \operatorname{ad}_{H^{(1)}}) ] (\rho_s^\gamma) \right\|\right)ds\leq \frac{C}{\gamma} \ .$$
	
	All that remains is to deal with the last term by proving that:
\begin{align}
\label{eq:decoherence_term_3}
\exists \ C>0, \quad 
\E\left(\left\|\int_0^t  [\Pi_\perp \circ  T_{s,t}^\gamma] \left( \sigma^{(0)} \ (\rho_s^\gamma) \cdot dW_s^0 \right) \right\|\right)
\leq
\
\frac{C}{\gamma}
 ,
\end{align}
which is the most delicate. Notice that the estimate needs to be uniform in $t \in \R_+$.
	
\medskip

{\bf Step 3: Proof of Eq.~\eqref{eq:decoherence_term_3} and conclusion.}
    We start by a sequence of reductions. First, by equivalence of the norms, it is sufficient to prove that for any couple of distinct $i,j$ in $\{1,\dotsc,d\}$, there exists $C>0$ such that for any $\gamma>0$ and $t\in \R_+$,
\begin{align*}
\E\left(\left|\left\langle E_{i,j}, \int_0^t  [T_{0,t}^\gamma \circ  (T_{0,s}^\gamma)^{-1}]  \left( \sigma^{(0)} (\rho_s^\gamma) \cdot dW_s^0 \right) \right\rangle\right|\right)\leq\frac{C}{\gamma} \ .
\end{align*}

Second, we reduce further by applying the triangular inequality and separating the different components of
$$ \sigma^{(0)} (\rho_s^\gamma) \cdot dW_s^0
   =
   \sum_{k=1}^{\ell_0} \sigma^{(0)}_k (\rho_s^\gamma)  \ dW_s^{0, k} \ ,
$$
which are made explicit in Eq.~\eqref{eq:def_sigma}. In the end, we only need to prove that for a one-dimensional Brownian motion $W$ and a matrix process $c^\gamma$, uniformly bounded in $\gamma$, the following holds for $i \neq j$:
\begin{align*}
  & \ \E\left(\left|\left\langle E_{i,j}, \int_0^t  [T_{0,t}^\gamma \circ (T_{0,s}^\gamma)^{-1}]( c^\gamma_s ) \ dW_s \right\rangle\right|\right)\\
= & \ \E\left(\left| [T_{0,t}^\gamma]_{i,j} \int_0^t ([T_{0,s}^\gamma]_{i,j})^{-1} [c^\gamma_s]_{i,j} \ dW_s \right|\right)\\
\leq & \ \frac{C}{\gamma} \ .
\end{align*}

Third, recalling the expression \eqref{eq:expression_T} of $[T_{0,t}^\gamma]_{i,j}$ from Proposition \ref{proposition:traj_decomposition}, by packaging all the stochastic integrals in the Doléans--Dade exponential into a single one, we can write:
$$ [T_{0,t}^\gamma]_{i,j} = e^{-\frac{t}{2} D^{\gamma}_{i,j} + \imath \theta_{i,j}^\gamma(t) }
   \Ec\left( \gamma \int_0^t a^\gamma_s dB_s \right) \ ,
$$
for a certain Brownian motion $B$ independent of $W$, and a real process $a^\gamma$, uniformly bounded in $\gamma$.

All in all, absorbing the phase $e^{\imath \theta_{i,j}^\gamma(t)}$ in $[c^\gamma_s]_{i,j}$ and considering separately real and imaginary parts, we need to prove 
\begin{lemma}
\label{lemma:lyapounov}
Let $W$ and $B$ be two independent Brownian motions. For any real processes $a^\gamma$ and $b^\gamma$, universally bounded:
$$ \sup_{\gamma > 1} \left( \|a^\gamma\|_\infty + \|b^\gamma\|_\infty \right)
   < \infty
$$
and any constants $D^\gamma$ such that  
\begin{equation}
\label{eq:lowerboundD}
\inf_{\gamma >0} \; \gamma^{-2} D^\gamma >0 \ ,
\end{equation} 
we have:
\begin{align*}
       \ \sup_{t \in \R_+} \E\left(\left| Z_t^\gamma \int_0^t e^{-\half (t-s) D^\gamma }(Z_s^\gamma)^{-1} b^\gamma_s \ dW_s \right|\right)
\leq & \ \frac{C}{\gamma} \ ,
\end{align*}
where 
$$ Z_t^\gamma := \Ec\left( \gamma \int_0^t a^\gamma_s dB_s \right) \ .$$
\end{lemma}
For improved readability, the proof of this lemma is given separately in the next subsection, and we are formally done with the proof of Lemma \ref{lemma:tightness_offdiag}.
\end{proof}

We go on with the proof of tightness of the process corresponding to the diagonal elements of $\rho^\gamma$, in the Meyer-Zheng topology. To that endeavor, let us recall a convenient tightness criterion.

\begin{thm} \cite[Theorem 4]{MZ}
\label{thm:meyer_zheng}
Let $E$ be a Euclidean space. If $X$ is a $E$-valued stochastic process with natural filtration $\left( \Fc_t \ ; \ t \geq 0 \right)$, then for any $\tau \in \R_+$, its conditional variation on $[0,\tau]$ is defined as:
\begin{align}
\label{def:conditional_variation}
V_\tau(X) := & \sup_{0=t_0<t_1<\dotsb<t_k=\tau}\; 
               \sum_{i=0}^{k-1}\E \Big(\|\E(X_{t_{i+1}}-X_{t_{i}}|\Fc_{t_{i}})\| \Big) \ .
\end{align}

Consider an index set $I$ and a family $\left( X^{(n)} \ ; \ n \in I\right)$ of processes living in $\D({\R}_+ ; E)$ which satisfy
$$ \sup_{n} \left[ V_\tau(X^{(n)}) +  \E\left[ \sup_{0 \leq t \leq \tau} X_t^{(n)} \right] \right] < \infty \ ,$$
for all $\tau>0$. Then the family of laws of the $X^{(n)}$'s is tight for the Meyer-Zheng topology and all limiting points are supported in $\D({\mathbb R}_+ ; E)$.
\end{thm}

The previous theorem is the main tool for proving:
\begin{lemma}
\label{lemma:tightness_diag}
The family of processes 
$$ \left( \Pi (\rho^\gamma) \ ; \ \gamma \geq 1 \right) $$ 
is tight in the Polish space $({\mathbb L}^0 (\R_+; {\mathcal S}), {\rm d})$ and all limiting points are supported on c\`adl\`ag paths.
\end{lemma}
\begin{proof}
Since for any $\gamma>0$ and any $t\geq 0$, $\rho_t^\gamma\in \Sc$ almost surely and $\sup_{\rho\in\Sc}\|\rho\|=1$, almost surely, $\Pi (\rho^\gamma)$ takes value in the Hilbert-Schmidt centered unit ball and therefore in a compact set. As such, only the conditional variation in Meyer-Zheng's Theorem \ref{thm:meyer_zheng} needs to be uniformly bounded on segments. For fixed $\tau>0$, we have in our case
\begin{align}
\label{eq:Vtau}
V_\tau(\Pi (\rho^\gamma)) := & \sup_{0=t_0<t_1<\dotsb<t_k=\tau}\sum_{i=0}^{k-1}\E(\|\E(\Pi (\rho^\gamma_{t_{i+1}})-\Pi (\rho^\gamma_{t_{i}}) |\Fc_{t_{i}})\|).
\end{align}
Equivalently, thanks to \cite[Eq.~(4), (5)]{MZ} and the following paragraph,
$$ V_\tau(\Pi(\rho^\gamma))
   =
   \sup_{\|\varphi\| \leq 1} \ \int_0^\tau \E \left[ \left\langle \varphi_t,\Pi ( d\rho^\gamma_t )  \right\rangle \right]$$
where the supremum is taken over the simple predictable processes taking value in the Hilbert--Schmidt centered unit ball of $M_d(\C)$. It follows from Eq.~\eqref{eq:SDE} that
$$V_\tau(\Pi ( \rho^\gamma ) )=\sup_{\Vert \varphi \Vert \le 1} \ \int_0^\tau \E\left[ \left\langle \varphi_t, [\Pi \circ  \Lc_\gamma]  (\rho^\gamma_t) \right\rangle \right]dt.$$
Since $t\mapsto [\Pi \circ \Lc_\gamma]  (\rho_t^\gamma) /\| [\Pi\circ \Lc_\gamma]  (\rho_t^\gamma)\|$ is adapted and locally square integrable, it can be approximated by simple predictable processes and we deduce that,
$$V_\tau(\Pi (\rho^\gamma))=\int_0^\tau \E(\| [\Pi \circ  \Lc_\gamma]  (\rho^\gamma_t)\|)dt.$$
By assumption $\Pi \Lc^{(2)} = \Pi \Lc^{(1)} \Pi=0$, hence by the triangular inequality,
$$ V_\tau(\Pi (\rho^\gamma) )
  \leq
  \int_0^\tau \E(\| [\Pi \circ  \Lc^{(0)}] \ (\rho^\gamma_t)\|) dt
  +
  \gamma \int_0^\tau \E \left(\left\| [\Pi \circ \Lc^{(1)} \circ \Pi_\perp ] (\rho_t^\gamma)\right\|\right) dt \ .$$
The operators $\Lc^{(0)}$ and $\Lc^{(1)}$ being bounded and $\rho_t^\gamma\in\mathcal S$ almost surely for any $\gamma>0$ and $t\in\R_+$, there exists $C>0$ such that,
\begin{align*}
    V_\tau(\Pi(\rho^\gamma))
    \leq & \ \ C \left( \tau+ \gamma \int_0^\tau \E( \|\Pi_\perp (\rho_t^\gamma)\|)  dt \right)\\
    \stackrel{\textrm{Eq.~\eqref{eq:offdiag_L1_sqrtgamma}}}{\leq} &
    C \left( \tau+ \gamma \int_0^\tau \left( e^{-g \gamma^2 t} + \frac{1}{\gamma} \right) dt \right)\\
    \leq & \ \ 
    C \left( 2 \tau + \frac{1}{\gamma g} \right) \ .
\end{align*}
This bound yields the result.
\end{proof}

\subsection{Proof of Lemma \ref{lemma:lyapounov}}
\label{subsection:gronwall}
	The proof is based on a change of measure,  Burkholder-Davis-Gundy  inequality for $p=1$ and two Gr\"onwall inequalities. We denote $\mathbb P$ the measure with respect to which $B$ and $W$ are independent multidimensional Brownian motions. Let $\F := \left( \mathcal F_t \ ; \ t \geq 0\right)$ be the natural filtration associated to the processes $W$, $B$, $a^\gamma$ and $b^\gamma$.
	
	\medskip
	
	{\bf Step 1: BDG inequalities}
	
	By assumption, for each $\gamma>1$, $Z^\gamma$ verifies Novikov's condition, therefore $Z^\gamma$ is a martingale with respect to $\mathbb P$. Let $\Q^\gamma$ be the probability measure defined from the Radon-Nikodym derivative $Z_t^\gamma = \frac{d\Q^\gamma}{d\P} {\vert_{\Fc_t}}$.
	
	By Girsanov's theorem, $W$ is still a Brownian motion with respect to $\Q^\gamma$ since it is independent of $B$ with respect to $\mathbb P$. Moreover, the boundedness of the processes $a^\gamma$ and $b^\gamma$ is unchanged under $\Q^\gamma$. In the following we denote by $\E$ the expectation with respect to $\mathbb P$ and $\E^{\mathbb Q^\gamma}$ the one with respect to $\mathbb Q^\gamma$. Using the change of measure, then the Burkholder-Davis-Gundy (BDG) inequality for $p=1$ \cite[Chapter IV, Theorem 4.1]{revuz2013continuous}, we have for some universal constant $C_1>0$:
	\begin{align*}
  & \ \E\left(\left| Z_t^\gamma \int_0^t e^{-\half (t-s) D^\gamma }(Z_s^\gamma)^{-1} b^\gamma_s \ dW_s \right|\right)\\
= & \  \E^{\Q^\gamma} \left(\left| \int_0^t e^{-\half (t-s) D^\gamma }(Z_s^\gamma)^{-1} b^\gamma_s \ dW_s \right|\right)\\
\leq & C_1
       \E^{\Q^\gamma}\left( \left| \int_0^t e^{-(t-s) D^\gamma }(Z_s^\gamma)^{-2} (b^\gamma_s)^2 \ ds \right|^\half \right) \ .
	\end{align*}
	Hence, reverting back to the expectation with respect to $\mathbb P$, we have:
	\begin{equation}
	\label{eq:previousstep}
	\begin{split}
  & \ \E\left(\left| Z_t^\gamma \int_0^t e^{-\half (t-s) D^\gamma }(Z_s^\gamma)^{-1} b^\gamma_s \ dW_s \right|\right)\\ 
\leq  & \ C_1
       \E\left( Z_t^\gamma \left| \int_0^t e^{-(t-s) D^\gamma }(Z_s^\gamma)^{-2} (b^\gamma_s)^2 \ ds \right|^\half \right) \\ 
\leq & \ C_1 \ \| b^\gamma \|_\infty \ \E\left( Z_t^\gamma \left| \int_0^t e^{-(t-s) D^\gamma }(Z_s^\gamma)^{-2} \ ds \right|^\half \right) , 
\end{split}	
	\end{equation}
	where in the last step, we invoked the fact that $b^\gamma$ is bounded.

	\medskip
	
	{\bf Step 2: Reductions.}
	
	Now, we shall reduce the problem to proving
	\begin{align}
	\label{eq:step2_reduction}
	\sup_{\substack{\gamma>1 \\ t \geq 1}}\; 
	\E\left( M_{t}^\gamma \left| \int_0^{t} e^{-(t-s) }(M_{s}^\gamma)^{-2} \ ds \right|^\half \right)
	< \infty
	\end{align}
	for all exponential martingale $M^\gamma$ in the form
	$$
	M_{t}^\gamma = \Ec\left( \int_0^{t} a_s^\gamma dB_s \right) 
	\quad
	\textrm{with }\quad 
	\sup_{\gamma > 1} \|a^\gamma\|_\infty < \infty \ .
	$$

    Starting from the last line of \eqref{eq:previousstep}, we perform a change time scale from $t$ to $D^\gamma t$:
	\begin{align*}
  & \ \E\left(\left| Z_t^\gamma \int_0^t e^{-\half (t-s) D^\gamma }(Z_s^\gamma)^{-1} b^\gamma_s \ dW_s \right|\right)\\
\leq & \ \frac{\ C_1 \ \| b^\gamma \|_\infty}{\sqrt{D^\gamma}} \; \E\left( Z_{D^\gamma t / D^\gamma}^\gamma \left| \int_0^{D^\gamma t} e^{- (D^\gamma t-s) }(Z_{s /  D^\gamma}^\gamma)^{-2} \ ds \right|^\half \right) ,
	\end{align*}
	Thanks to the assumption \eqref{eq:lowerboundD} it remains to prove
	$$ \sup_{\substack{\gamma>1 \\ t\in\R_+}}
	\E\left( Z_{t / D^\gamma}^\gamma \left| \int_0^{t} e^{-(t-s) }(Z_{s /  D^\gamma}^\gamma)^{-2} \ ds \right|^\half \right)
	< \infty \ .
	$$
	Upon writing ${B}_{s/D^\gamma} = (D^\gamma)^{-\half } \widetilde{B}_{s}$ and $a^\gamma_s = \gamma (D^\gamma)^{-\half } \widetilde{a}^\gamma_{s/D^\gamma}$, we have:
	\begin{align*}
	  Z_{t / D^\gamma}^\gamma = & \ \Ec\left( \gamma \int_0^{t/D^\gamma} a_s^\gamma dB_s \right) =  \ \Ec\left( \gamma (D^\gamma)^{-\half } \int_0^{t} a_{s/D^\gamma}^\gamma d\widetilde{B}_s \right)\\
	= & \ \Ec\left( \int_0^{t} \widetilde{a_s^\gamma} \ d\widetilde{B}_s \right)
	 ,
    \end{align*}
	where $\widetilde{B}$ is still a Brownian motion thanks to scale invariance. Now, notice that thanks to the assumption \eqref{eq:lowerboundD}, we have:
	$$ \sup_{\gamma > 1} \|\widetilde{a}^\gamma\|_\infty < \infty \ .$$
As such, upon renaming variables and processes, i.e. $\widetilde{B}=B$ and $\widetilde{a^\gamma}=a^\gamma$, we can set $M_t^\gamma := \Ec\left( \int_0^t a^\gamma_s dB_s \right)$, and we see that we only need to prove that:
	$$ \sup_{\substack{\gamma>1 \\ t\in\R_+}}
	\E\left( M_{t}^\gamma \left| \int_0^{t} e^{-(t-s) }(M_{s}^\gamma)^{-2} \ ds \right|^\half \right)
	< \infty \ .
	$$
	In order to recover Eq.~\eqref{eq:step2_reduction}, we need to discard the supremum over $t \in [0,1]$. This is achieved with Cauchy-Schwarz's inequality and then Fubini's theorem:
\begin{align*}
     & \sup_{\substack{\gamma>1 \\ t\in [0,1]}}
	   \E\left( M_{t}^\gamma \left| \int_0^{t} e^{-(t-s) }(M_{s}^\gamma)^{-2} \ ds \right|^\half \right)\\
\leq & \sup_{\substack{\gamma>1 \\ t\in [0,1]}}
	   \E\left( (M_{t}^\gamma)^2 \int_0^{t} e^{-(t-s) }(M_{s}^\gamma)^{-2} \ ds \right)^\half\\
\leq & \sup_{\substack{\gamma>1 \\ (s,t) \in [0,1]^2}} \E\left( (M_{t}^\gamma)^2 (M_{s}^\gamma)^{-2} \right)^\half
	< \infty \ .
\end{align*}	
	Indeed, since $\sup_{\gamma>0} \Vert a^\gamma \Vert_{\infty} <\infty$, the random variables $M^\gamma_t$ and $(M_t^\gamma)^{-1}$ have all their moments uniformly bounded in $\gamma>1$ and $t \in [0,1]$.
	
    \medskip
	
	{\bf Step 3: Conclusion.}
	We conclude by studying the process $Y^\gamma$ defined by
    $$
       Y_t^\gamma := \ M_{t}^\gamma \left| \int_0^{t} e^{-(t-s) }(M_{s}^\gamma)^{-2} \ ds \right|^\half \ ,
    $$ 
    for $t \geq 1$ and proving the claim in Eq.~\eqref{eq:step2_reduction} thanks to a Lyapounov function argument.
    
    Applying Itô's formula to $Y^\gamma_t$ and then to $(Y^\gamma_t)^{-1}$, we have:
    \begin{align*}
    dY_t^\gamma & = \frac{d(M^\gamma_t e^{-\half t})}{M^\gamma_t e^{-\half t}} Y_t
                  + \frac{M^\gamma_t e^{-\half t} e^t (M^\gamma_t)^{-2}}
                         {2 \sqrt{\int_0^{t} e^{s }(M_{s}^\gamma)^{-2} \ ds}} dt\\
                & = \frac{dM^\gamma_t}{M^\gamma_t} Y_t^\gamma
                  + \frac{d (e^{-\half t})}{e^{-\half t}} Y_t^\gamma
                  + \frac{ 1 }
                         {2 M^\gamma_t e^{-\half t} \sqrt{\int_0^{t} e^{s }(M_{s}^\gamma)^{-2} \ ds}} dt\\
                & = a_t^\gamma Y_t^\gamma \ dB_t
                  - \half Y_t^\gamma dt
                  + \frac{ dt }
                         {2 Y_t^\gamma } \ ,
    \end{align*}
    and then
    \begin{align*}
    d(Y_t^\gamma)^{-1} & = - \frac{dY_t^\gamma}{(Y_t^\gamma)^2}
                           + \half 2 \frac{d\langle Y^\gamma, Y^\gamma \rangle_t}{(Y_t^\gamma)^3}\\
& = - \left(Y_t^\gamma\right)^{-2} \left( a_t^\gamma \  Y_t^\gamma \ dB_t
                  - \half Y_t^\gamma dt
                  + \frac{ dt }
                         {2 Y_t^\gamma } \right)  
    + \frac{|a^\gamma_t|^2}{Y_t^\gamma} dt\\
& = - a_t^\gamma (Y_t^\gamma)^{-1} \ dB_t
    + \half (Y_t^\gamma)^{-1} dt 
    - \half (Y_t^\gamma)^{-3} dt 
    + |a^\gamma_t|^2 (Y_t^\gamma)^{-1} dt \ .
\end{align*}
    Taking expectations, we have:
	\begin{align}
	\label{eq:ode_system}
	\left\{
	\begin{array}{ccc}
		\cfrac{d}{dt}\ \E[Y_t^\gamma] & = &-\half \E [Y_t^\gamma]+\half \E[(Y_t^\gamma)^{-1}] \ ,\\
		                       &   & \\
		\cfrac{d}{dt}\ \E[(Y_t^\gamma)^{-1}]& = & \E\left[(\half + |a^\gamma_t|^2 )(Y_t^\gamma)^{-1}\right]-\frac12\E[(Y_t^\gamma)^{-3}]  \ .
	\end{array}
	\right.
	\end{align}
	 Now, one can show that:
     $$ \forall \alpha>0,\quad \forall x>0, \quad
        x^{-3}\geq \alpha x^{-1}- 2\Big(\frac{\alpha}3\Big)^{\frac32} \ .     
     $$
     Therefore, for any $\gamma>1$, $\alpha>0$:
	$$\frac{d}{dt} \E[(Y_t^\gamma)^{-1}]
	  \leq
	  \E\left[ \left(\half + |a^\gamma_t|^2  -\half \alpha \right)(Y_t^\gamma)^{-1} \right]
	 +\Big(\frac\alpha 3\Big)^{\frac32} \ .$$
	Setting $\alpha = 1 + 2(1 + \sup_{\gamma > 1}\|a^\gamma\|_\infty)$ so that $\half + |a^\gamma_t|^2  -\half \alpha \leq -1$ almost surely, we find:
	$$\forall t \geq 1, \quad 
	  \frac{d}{dt}\E[(Y_t^\gamma)^{-1}]
	  \leq
	  - \E\left[ (Y_t^\gamma)^{-1} \right]
	  + \Big(\frac\alpha 3\Big)^{\frac32}.$$

Using Gr\"onwall's inequality, there exists $C>0$ such that for all $t \geq 1$:
	$$\E[(Y_t^\gamma)^{-1}] \leq \E[(Y_1^\gamma)^{-1}] + C \ .$$
Then:
\begin{align*}
     \ \E\left[ (Y_1^\gamma)^{-1} \right]
 & = \ \E\left[ (M_{1}^\gamma)^{-1}
                \left| \int_0^{1} e^{-(1-s) }(M_{s}^\gamma)^{-2} \ ds \right|^{-\half} \right] \\
 & \leq \ \E\left[ (M_{1}^\gamma)^{-1} \right]^\half
        \ \E\left[ \left| \int_0^{1} e^{-(1-s) }(M_{s}^\gamma)^{-2} \ ds \right|^{-1} \right]^\half \\
 & \leq \ e^{\half}
        \ \E\left[ (M_{1}^\gamma)^{-1} \right]^\half
        \ \E\left[ \left| \inf_{0 \leq s \leq 1} (M_{s}^\gamma)^{-2} \right|^{-1} \right]^\half \\
 & \leq \ e^{\half}
        \ \E\left[ (M_{1}^\gamma)^{-1} \right]^\half
        \ \E\left[ \sup_{0 \leq s \leq 1} (M_{s}^\gamma)^{2} \right]^\half \ .
\end{align*}
This expression is universally bounded because of Doob's inequality and the properties of the martingale $M^\gamma$. We have thus proven that there exists $C>0$ such that for all $t \geq 1$:
	$$\E[(Y_t^\gamma)^{-1}] \leq 2C \ .$$
Injecting the above equation in the first equation in Eq.~\eqref{eq:ode_system}, we find:
$$ \frac{d\E(Y_t^\gamma)}{dt} = -\half \E(Y_t^\gamma) + C \ .$$
	Using again Gr\"onwall's inequality and the fact that $\E (Y_1^\gamma)$ is universally bounded, we find that indeed
	$$\sup_{\gamma >1, t \geq 1} \E(Y_t^\gamma) < \infty \ .$$
	The lemma is thus proven.

\subsection{The time spent away from the pointer states vanishes}
\label{subsection:time_control}

\begin{lemma}\label{lemma:Loc}
Let $(\gamma_n)_n$ be an unbounded sequence in $\R_+$ such that $(\rho^{\gamma_n})_n$ converges weakly in $ \left( {\mathbb L}^0 (\R_+;M_d(\C)), {\rm d} \right)$ to some random measurable function $\beta$. Then, for almost every $t\in\R_+$, almost surely, $\beta_t\in \left( E_{i,i} \right)_{i=1}^d$.

More precisely, let $\B\left( E_{i,i}, \varepsilon \right)$ denote the Hilbert--Schmidt ball in $M_d(\C)$ of radius $\varepsilon>0$ centered in $E_{i,i}$. Let $T_\varepsilon^\gamma(t)$ be the time spent outside of any of the balls $\B\left( E_{i,i}, \varepsilon \right)$ up to time $t$: 
$$ T_\varepsilon^\gamma(t) := \int_0^t \mathds{1}_{\{ \rho_s^\gamma\notin \cup_{i=1}^d \B\left( E_{i,i}, \varepsilon \right) \} } \ ds \ .$$
Then, for any $\varepsilon>0$ there exists $C>0$ such that for any $t\geq 0$,
\begin{equation}\label{eq:bound_expectedtime}
\limsup_{\gamma\to\infty} \gamma \ \E\left( T_\varepsilon^\gamma(t) \right) \le Ct.
\end{equation}
\end{lemma}
\begin{proof}
We are interested in the quadratic variation of the noise with leading order in $\gamma$ in the SDE \eqref{eq:SDE}, when restricting to the diagonal. Thanks to From Eq.~\eqref{eq:def_sigma} for $i=2$, and thanks to the QND Assumption \ref{ass:(Nd)}, it is computed as follows:
\begin{align*}
     \left\| \Pi\sigma^{(2)} (\rho) \right\|^2
 = & \sum_{k=1}^{\ell_2}
     \left\| \Pi(\sigma^{(2)} (\rho))_k \right\|^2 \quad \textrm{(definition)}\\
 = & \sum_{k=1}^{\ell_2}
     \eta_2(k) \left\| \Pi  (L^{(2)}_k \rho+\rho L^{(2)\, *}_k-\tr[(L^{(2)\, *}_k+L^{(2)}_k)\rho]\rho) \right\|^2\\
 = & \sum_{k=1}^{\ell_2}
     \eta_2(k) \left\| \Pi \rho \left( (L^{(2)}_k + L^{(2)\, *}_k)-\tr[(L^{(2)\, *}_k+L^{(2)}_k)\rho] \right) \right\|^2\\
 = & \ 4 \sum_{k=1}^{\ell_2}
     \sum_{i=1}^d
     \eta_2(k) \left(\Re ( L^{(2)}_k )_{i,i} - R_k (\rho) \right)^2\rho_{i,i}^2\\
     =:& \ f(\rho)
\end{align*}
with $R_k (\rho) = \sum_{i=1}^d  \Re ( L^{(2)}_k )_{i,i} \ \rho_{i,i}$. Then, $f(\rho)=0$ is equivalent to
\begin{align*}
& \forall k\in\{1,\dotsc,\ell\},
       \ \forall i\in\{1,\dotsc,d\}, 
       \ \eta_2(k) \left( \Re ( L^{(2)}_k )_{i,i} - R_k (\rho) \right) \rho_{i,i}=0
\end{align*}
which in turn is equivalent to
\begin{align*}
& \forall k\in\{1,\dotsc,\ell\},
       \ \forall i,j\in\{1,\dotsc,d\},
       \ \eta_2(k)\rho_{i,i}\rho_{j,j}\left( \Re ( L^{(2)}_k )_{i,i} - \Re ( L^{(2)}_k )_{j,j} \right) = 0.
\end{align*}
The identifiability Assumption \ref{ass:(Id)} thus implies $\rho_{i,i}\rho_{j,j}=0$ for any $i\neq j$. Hence, there exists at most one $i\in \{1,\dotsc,d\}$ such that $\rho_{i,i}\neq0$.
Since $\rho\in\Sc$, $\rho_{i,i}\geq 0$ for any $i\in\{1,\dotsc,d\}$ and $\sum_{i=1}^d \rho_{i,i}=1$,
$$ f(\rho) = 0  \iff \rho\in \left( E_{i,i} \right)_{i=1}^{d} \ .$$
In particular, continuity implies that for any small enough $\varepsilon>0$, there exists a constant $K>0$ such that 
\begin{equation}\label{eq:f_lower_bound}
f(\rho)\geq K\mathds{1}_{\left\{\rho\notin \cup_{i=1}^d \mathbb B(E_{i,i},\varepsilon) \right\}}.
\end{equation}

Since $\rho\mapsto \|\Pi (\rho)\|^2$ is $C^2$, It\^o calculus implies
\begin{equation}
\begin{split}
\label{eq:diff_rho_0_rho_t_squarred}
   \E(\|\Pi(\rho_t^{\gamma}-\rho_0^{\gamma})\|^2)
= & \ \E\left(\int_0^t2\Re(\langle [\Pi \circ \Lc_{\gamma}] (\rho_s^{\gamma}),\Pi (\rho_s^{\gamma}) \rangle) \right.\\
& \ \quad \quad  \left. \phantom{\int_0^t}  +\| [\Pi\circ \sigma^{(0)}] (\rho_s^{\gamma})\|^2 + \gamma\| [\Pi \circ \sigma^{(1)}] (\rho_s^{\gamma})\|^2) ds  \right)  \\
  & \ \quad \quad +\gamma^2\E\left(\int_0^tf(\rho_s^{\gamma})d s  \right) \\
  = &\  \gamma^2 \E\left(\int_0^tf(\rho_s^{\gamma})d s \right) +\Oc(\gamma t) \ ,
\end{split}
\end{equation}
since $\Pi\Lc^{(2)}=0$ and $\Sc$ is compact. Dividing both sides of Eq.~\eqref{eq:diff_rho_0_rho_t_squarred} by $\gamma$, we deduce that
\begin{equation}\label{eq:limsup_expected f}
\limsup_{\gamma\to\infty} \gamma\int_0^t\E(f(\rho_s^{\gamma}))d s\leq Ct \ .
\end{equation}
Eq.~\eqref{eq:bound_expectedtime} follows then from Eq.~\eqref{eq:f_lower_bound}.

Moreover, the continuity of $f$, the weak convergence of the sequence $\left( \rho^{\gamma_n} \right)_{n \geq 1}$ taking values in $\Bc(\R_+;M_d(\C))$ and Eq.~\eqref{eq:limsup_expected f} imply 
$$ \lim_{n \rightarrow \infty} \E \left( \int_0^t f(\rho_s^{\gamma_n})d s \right)
   =
   \E\left( \int_0^t f(\beta_s)ds \right) = 0 \ .$$
The function $f$ being non negative and vanishing only on $\left( E_{i,i} \right)_{i=1}^{d}$, it implies that for almost every time $t \in \R_+$, almost surely, $\beta_t \in \left( E_{i,i} \right)_{i=1}^d$.
\end{proof}

\subsection{Finite dimensional distributions}
\label{subsection:fd_distributions}
Proposition \ref{proposition:tightness} proves tightness of the family $\left( \rho^\gamma \ ; \ \gamma \geq 0 \right)$ in $ \left( {\mathbb L}^0 (\R_+;M_d(\C)), {\rm d} \right)$. It only remains to show that any weakly convergent sequence converges to the same random variable in ${\mathbb L}^0 (\R_+;M_d(\C))$. The idea behind the proof is to use Lemma \ref{lemma:Loc} to reduce continuous functions to linear ones. Then we only have to rely on the mean convergence from Proposition \ref{proposition:Lindblad_perturbation} in order to find the limiting finite-dimensional distribution.

Let $(\gamma_n)_n$ be an unbounded sequence in $(1,+\infty)$ such that $(\rho^{\gamma_n})$ converges weakly in ${\mathbb L}^0 (\R_+;M_d(\C))$ to some $\beta$. 
Let $f$ be a continuous, and therefore bounded function of $\Sc^r$ for a fixed integer $r \in \N$. The goal is to characterize the expectation
$$ \E\left[ f\left( \beta_{t_1}, \beta_{t_2}, \dots, \beta_{t_r} \right) \right]$$
for every $r \in \N$ and for almost every $t_1, \dots, t_r$ in $\R_+$.

\medskip

{\bf Step 1: An $r$-linearization trick}

Write $F_{\bf i} := f(E_{i_1,i_1},\dotsc,E_{i_r,i_r})$ for any $\mathbf{i}=(i_1,\dotsc,i_r)\in\{1,\dotsb,d\}^r$ and notice that if $(\rho_1,\dotsc,\rho_r)\in\{E_{i,i}: i=1,\dotsc, d\}^r$,
\begin{align*}
     \ f(\rho_1,\dotsc,\rho_r)
=  & \ \sum_{\mathbf{i}\in\{1,\dotsc,d\}^r}
      F_{\mathbf{i}} \prod_{k=1}^r \mathds{1}_{E_{i_k,i_k}}(\rho_k)\\
=  & \ \sum_{\mathbf{i}\in\{1,\dotsc,d\}^r}F_{\mathbf{i}} \prod_{k=1}^r (\rho_k)_{i_k,i_k}\\
=: & \ F(\rho_1,\dotsc,\rho_r) \ .
\end{align*}
This latter function $F$ is a continuous $r$-linear map, and shall be referred to as the $r$-linearization of $f$. Now \cite[Theorem 6]{MZ} applied to $f$ and $F$ says exactly that
\begin{equation}\label{eq:conv_to_beta}
\lim_n\int_{[0,\infty)^{r}}\left|\E(f(\rho_{t_1}^{\gamma_n},\dotsc,\rho_{t_r}^{\gamma_n}))-\E(f(\beta_{t_1},\dotsc, \beta_{t_r}))\right|\lambda^{\otimes r}(dt_1\dotsc dt_r)=0
\end{equation}
and
$$ \lim_n\int_{[0,\infty)^{r}}\left|\E(F(\rho_{t_1}^{\gamma_n},\dotsc,\rho_{t_r}^{\gamma_n}))-\E(F(\beta_{t_1},\dotsc, \beta_{t_r}))\right|\lambda^{\otimes r}(dt_1\dotsc dt_r)=0 \ .$$
Lemma \ref{lemma:Loc} implies that for almost every $r$-tuple of times $(t_1,\dotsc, t_r)$, almost surely:
$$ f\left( \beta_{t_1}, \dotsc, \beta_{t_r} \right)
 = F\left( \beta_{t_1}, \dotsc, \beta_{t_r} \right) \ .
$$
It follows that, in the large $n$ limit, every continuous function $f$ in $r$ variables can be replaced by its $r$-linearization:
\begin{equation}
\label{eq:reduction_to_linear}
\lim_n \int_{[0,\infty)^{r}}\left|\E(f(\rho_{t_1}^{\gamma_n},\dotsc,\rho_{t_r}^{\gamma_n}))- \E(F(\rho_{t_1}^{\gamma_n},\dotsc,\rho_{t_r}^{\gamma_n}))\right|\lambda^{\otimes r}(dt_1\dotsc dt_r)=0.
\end{equation}

{\bf Step 2: Markov property between pointer states}

Let $P_{i,j}^\gamma(t) := (e^{t\Lc_\gamma}(E_{i,i}))_{j,j}$ be transition rates between pointer states only. We claim that for any $r \in \N$ and any $\mathbf{i}\in\{1,\dotsc,d\}^r$, 

\begin{align}
\label{eq:limit_multi_d_linearmaps}
0 = & \ \lim_{n \rightarrow \infty} \int \displaylimits_{\substack{0=t_0\leq t_1\leq \dotsb \leq t_r<\infty}}
      \Big| \ \E(\prod_{k=1}^r (\rho_{t_k}^{\gamma_n})_{i_k,i_k}) \ -
      \sum_{i_0=1}^d\varrho_{i_0,i_0}\prod_{k=1}^r P_{i_{k-1}i_k}^{\gamma_n}(t_k-t_{k-1})\Big| 
     \lambda^{\otimes r}(dt_1\dotsc dt_r).
\end{align}
We prove this convergence by induction. 

For $r=1$, we start from 
$$ \E(\rho_t^\gamma)
 = e^{t\Lc_\gamma}(\varrho)
 = e^{t\Lc_\gamma}(\Pi \varrho) + o(1)
$$
since $\lim_{\gamma \rightarrow \infty} e^{t\Lc_\gamma}\Pi_\perp \varrho = 0$ by Proposition \ref{proposition:Lindblad_perturbation}.
Hence:
$$ \E( (\rho_t^\gamma)_{i_1, i_1} )
 = o(1) + \sum_{i_0=1}^d \varrho_{i_0, i_0} \left( e^{t\Lc_\gamma}( E_{i_0,i_0}) \right)_{i_1, i_1}
 = o(1) + \sum_{i_0=1}^d \varrho_{i_0, i_0} P^\gamma_{i_0, i_1}(t) \ .
$$
Integrating over $\R_+$ yields the claim.

Now, by induction hypothesis, assume that the claim holds for some $r\in\N$. Let $\left( \Fc_t ; \ t \geq 0 \right)$ be the natural filtration. For all $0\leq t_1 \leq \dotsb \leq t_r \leq t_{r+1} \leq \infty$, the tower property of conditional expectation and then the Markov property imply
\begin{align*}
        \E\left( \prod_{k=1}^{r+1} (\rho_{t_k}^{\gamma_n})_{i_k,i_k} \right)
 = & \  \E\left( \prod_{k=1}^r (\rho_{t_k}^{\gamma_n})_{i_k,i_k}
        \times 
        \E\left( (\rho_{t_{r+1}}^{\gamma_n} )_{i_{r+1},i_{r+1}} | \Fc_{t_r} \right)\right)\\
 = & \  \E\left( \prod_{k=1}^r (\rho_{t_k}^{\gamma_n})_{i_k,i_k}
        \times 
        (e^{(t_{r+1}-t_r) \Lc_{\gamma_n}} \rho_{t_{r}}^{\gamma_n} )_{i_{r+1},i_{r+1}} \right) \ .
\end{align*}
Introducing a comparison of $\rho_{t_r}^{\gamma_n}$ with $E_{i_ri_r}$,
\begin{align*}
   & \  \left|
        \E\left( \prod_{k=1}^{r+1} (\rho_{t_k}^{\gamma_n})_{i_k,i_k} \right)
        -
        \E\left( \prod_{k=1}^{r} (\rho_{t_k}^{\gamma_n})_{i_k,i_k} 
        \times
		P_{i_ri_{r+1}}^{\gamma_n}(t_{r+1}-t_r)        
        \right)
        \right|\\
 =  & \  \Big|
        \E\left( \prod_{k=1}^r (\rho_{t_k}^{\gamma_n})_{i_k,i_k}
        \times 
        (e^{(t_{r+1}-t_r) \Lc_{\gamma_n}} \rho_{t_{r}}^{\gamma_n} )_{i_{r+1},i_{r+1}} \right)\\
    & \ \qquad \qquad
        -
        \E\left( \prod_{k=1}^{r} (\rho_{t_k}^{\gamma_n})_{i_k,i_k} 
        \times
		P_{i_ri_{r+1}}^{\gamma_n}(t_{r+1}-t_r)        
        \right)
        \Big|\\
 =  & \ \left|
        \E\left[ \prod_{k=1}^r (\rho_{t_k}^{\gamma_n})_{i_k,i_k}
        \times
        \left(
        (e^{(t_{r+1}-t_r) \Lc_{\gamma_n}} \rho_{t_{r}}^{\gamma_n} )_{i_{r+1},i_{r+1}}
        -
		P_{i_ri_{r+1}}^{\gamma_n}(t_{r+1}-t_r)
		\right)
        \right]
        \right| \\
\leq  & \ \E\left[ \left| (\rho_{t_r}^{\gamma_n})_{i_r,i_r}
        \times
        \left(
        e^{(t_{r+1}-t_r) \Lc_{\gamma_n}} (\rho_{t_{r}}^{\gamma_n} - E_{i_r, i_r})
		\right)_{i_{r+1},i_{r+1}}
		\right|
        \right] \\
\leq  & \ \E\left[ \left| (\rho_{t_r}^{\gamma_n})_{i_r,i_r} \right|
        \
        \left\|
        \rho_{t_{r}}^{\gamma_n} - E_{i_r, i_r}
		\right\|
        \right] \ ,
\end{align*}
where the penultimate inequality follows from the fact that $\rho^\gamma$ is $\Sc$-valued. Upon integrating and taking the upper limit for $n \rightarrow \infty$, we have:
\begin{align*}
   & \  \limsup_{n \rightarrow \infty}
        \int \displaylimits_{\substack{0=t_0\leq t_1\leq \dotsb \leq t_r<\infty}}
        \Big|
        \E\left( \prod_{k=1}^{r+1} (\rho_{t_k}^{\gamma_n})_{i_k,i_k} \right)\\
   & \  \qquad \qquad 
        -
        \E\left( \prod_{k=1}^{r} (\rho_{t_k}^{\gamma_n})_{i_k,i_k} 
        \times
		P_{i_ri_{r+1}}^{\gamma_n}(t_{r+1}-t_r)        
        \right)
        \Big|
        \lambda^{\otimes r}(dt_1 \dots dt_r)\\
\leq  & \limsup_{n \rightarrow \infty} \int_0^\infty
        \ \E\left[ \left| (\rho_{t_r}^{\gamma_n})_{i_r,i_r} \right|
        \
        \left\|
        \rho_{t_{r}}^{\gamma_n} - E_{i_r, i_r}
		\right\|
        \right]
        \lambda(dt_r) \\
=     &  \limsup_{n \rightarrow \infty} \int_0^\infty
        \ \E\left[ \left| (\beta_{t_r})_{i_r,i_r} \right|
        \
        \left\|
        \beta_{t_{r}} - E_{i_r, i_r}
		\right\|
        \right]
        \lambda(dt_r) \\
=      & \ 0 ,
\end{align*}
where in the last step we invoked the fact that $\left| (\beta_{t_r})_{i_r,i_r} \right| = \mathds{1}_{E_{i_r,i_r}}(\beta_{t_r})$ for almost every $t_r$ (Lemma \ref{lemma:Loc}), so that the product with $\left\| \beta_{t_{r}} - E_{i_r, i_r} \right\|$ vanishes necessarily. Now invoking the induction hypothesis with $r$:
\begin{align*}
0 = \ & \  \limsup_{n \rightarrow \infty}
        \int \displaylimits_{\substack{0=t_0\leq t_1\leq \dotsb \leq t_r<\infty}}
        \Big|
        \E\left( \prod_{k=1}^{r} (\rho_{t_k}^{\gamma_n})_{i_k,i_k}
		P_{i_ri_{r+1}}^{\gamma_n}(t_{r+1}-t_r)        
		\right)\\
    \ & \qquad \qquad
        -
        \sum_{i_0=1}^d\varrho_{i_0,i_0}\prod_{k=1}^{r+1} P_{i_{k-1}i_k}^{\gamma_n}(t_k-t_{k-1})
        \Big|
        \lambda^{\otimes r}(dt_1 \dots dt_r) \ ,
\end{align*}
combined with the previous limit yields the claim for $r+1$.

\medskip

{\bf Step 3: Invoking the convergence of the mean.}
By linearity and Eq.~\eqref{eq:limit_multi_d_linearmaps}, we have for all $r \in \N$:
\begin{align}
\label{eq:limit_multi_d_linearmaps2}
\begin{split}
0 = & \ \lim_{n \rightarrow \infty} \int \displaylimits_{\substack{0=t_0\leq t_1\leq \dotsb \leq t_r<\infty}}
      \Big| \ \E(F(\beta_{t_1},\dotsc,\beta_{t_r})) \ - \\
    & \qquad \qquad \sum_{\mathbf i\in\{1,\dotsc,d\}^r} F_{\mathbf i}
                    \sum_{i_0=1}^d\varrho_{i_0,i_0}\prod_{k=1}^r P_{i_{k-1}i_k}^{\gamma_n}(t_k-t_{k-1})\Big| 
     \lambda^{\otimes r}(dt_1\dotsc dt_r).
\end{split}
\end{align}

By the convergence of the mean of Proposition \ref{proposition:Lindblad_perturbation}, $\lim_{\substack{\gamma\to\infty}} P_{i,j}^\gamma(t)=(e^{tT})_{i,j}$. Then the above Eq.~\eqref{eq:limit_multi_d_linearmaps2} together with Eq.~\eqref{eq:reduction_to_linear}, and \eqref{eq:conv_to_beta} imply that for $\lambda^{\otimes r}$-almost every $t_1 \leq t_2 \leq \dots \leq t_r$, we have:
\begin{align*}
  & \ \lim_n \E(f(\rho_{t_1}^{\gamma_n},\dotsc,\rho_{t_r}^{\gamma_n}))\\
= & \ \E(f(\beta_{t_1},\dotsc,\beta_{t_r}))\\
= & \ \E(F(\beta_{t_1},\dotsc,\beta_{t_r}))\\
= & \ \lim_n \E(F(\rho_{t_1}^{\gamma_n},\dotsc,\rho_{t_r}^{\gamma_n}))\\
= & \ \sum_{\mathbf i\in\{1,\dotsc,d\}^r} F_{\mathbf i} \sum_{i_0=1}^d \langle \varrho,E_{i_0i_0}\rangle \prod_{k=1}^r (e^{(t_k - t_{k-1})T})_{i_{k-1}i_k} \ .
\end{align*}

Then \cite[Theorem 6]{MZ} yields that $\beta$ has the same law as the process $\xb^* \xb$, where $\xb$ is the Markov process described in the statement of the Main Theorem \ref{thm:main} with generator $T$ on the pointer basis $\left( e_i \right)_{i=1}^d$. This concludes the proof.
\hfill \qed

\newcommand{\nocontentsline}[3]{}
\newcommand{\tocless}[2]{\bgroup\let\addcontentsline=\nocontentsline#1{#2}\egroup}
\section*{Acknowledgements}
T.B. would like to thanks Martin Fraas for enlightening discussions. The research of T.B., R.C.\ and C.P.\ has been supported by project QTraj (ANR-20-CE40-0024-01) of the French National Research Agency (ANR). The research of T.B.\ and C.P.\ has been supported by ANR-11-LABX-0040-CIMI within the program ANR-11-IDEX-0002-02. The work of C.B. and R.C. has been supported by the projects RETENU ANR-20-CE40-0005-01, LSD ANR-15-CE40-0020-01 of the French National Research Agency (ANR) and by the European Research Council (ERC) under  the European Union's Horizon 2020 research and innovative programme (grant agreement No 715734). J. N. has been supported by the start up grant U100560-109 from the University of Bristol and a Focused Research Grant from the Heibronn Institute for Mathematical Research. This work is supported by the 80 prime project StronQU of MITI-CNRS: ``Strong noise limit of stochastic processes and application of quantum systems out of equilibrium''.

\bibliographystyle{halpha}
\bibliography{biblio}

\end{document}